\documentclass[12pt]{article}

\usepackage{amsfonts}
\usepackage{ amsmath, amsthm}
\usepackage{amssymb, amsxtra}

\usepackage{graphicx}

\usepackage{multicol}

\usepackage[a4paper, text={.8\paperwidth,.8\paperheight}, ratio=1:1]{geometry}

\usepackage[colorlinks=true, linkcolor=blue,citecolor=blue]{hyperref}

\usepackage[english]{babel}
\usepackage{latexsym}

\usepackage{tikz}
\usetikzlibrary{calc}
\usepackage[tight]{subfigure}

\usepackage[labelsep=period]{caption}

\usepackage{color}
\definecolor{greenbean}{RGB}{199,237,204}

\newtheorem{thm}{Theorem}[section]
\newtheorem{lemma}[thm]{Lemma}
\newtheorem{prop}[thm]{Proposition}
\newtheorem{cor}[thm]{Corollary}

\newtheorem{Def}[thm]{Definition}

\newtheorem{rmk}[thm]{Remark}

\newcommand{\notion}[1]{\textit{#1}}
\newcommand{\Bigp}[1]{\Big(#1\Big)}
\title{Moduli Spaces of Arrangements of 10 Projective Lines with Quadruple Points \footnotetext{\hspace{-1.8em}2000 Mathematics Subject Classification. 14N20, 32S22, 52C35.\\ Key words and phrases: Line arrangements, moduli spaces, irreducibility
}}
\author{Meirav Amram, Mina Teicher, and Fei Ye}
\date{}

\begin{document}

\maketitle

\abstract{We classify moduli spaces of arrangements of 10 lines with quadruple points.  We show that moduli spaces of arrangements of 10 lines with quadruple points may consist of more than 2 disconnected components, namely 3 or 4 distinct points.  We also present defining equations to those arrangements whose moduli spaces are still reducible after taking quotients of complex conjugations.  }

\section{Introduction}

A {\it line arrangement} $\mathcal{A}$ 
in $\mathbb{CP}^2$ is a finite collection of projective lines.  The complement of the union of lines in $\mathcal{A}$ is denoted as $M(\mathcal{A})$ . We call the set $L(\mathcal{A})=\{\bigcap\limits_{i\in S}L_i | S\subseteq\{1, 2, \dots, n\}\}$ partially ordered by reverse inclusion the \notion{intersection lattice} of $\mathcal{A}$. 
Two line arrangements $\mathcal{A}_1$ and $\mathcal{A}_2$ are lattice isomorphic, denoted as $\mathcal{A}_1\sim \mathcal{A}_2$,  if their intersection lattices $L(\mathcal{A}_1)$ and $L(\mathcal{A}_2)$ are isomorphic, i.e., there is a permutation $\phi$ of $\{1, 2, \dots, n \}$ such that
\[\dim \Big( \bigcap\limits_{i\in S \atop L_i\in \mathcal{A}_1}L_i \Big) =\dim \Bigp{ \bigcap\limits_{j\in \phi(S)\atop H_j\in \mathcal{A}_2}H_j }\]
 for any nonempty subset $S\subseteq\{1,2,\dots, n\}$. In particular, if the permutation $\phi$ is the identity, we denote $L(\mathcal{A})=L(\mathcal{B})$.

An essential topic in hyperplane arrangements theory is to study the interaction between topology of complements and combinatorics of intersection lattices. One may ask how close topology and combinatorics of a given arrangement are related.

For line arrangements, Jiang and Yau \cite{Jiang1998} showed that  homeomorphic equivalence always implies lattice isomorphism.  On the other hand,
in 1989, Randell \cite{Randell1989} proved that if two arrangements are lattice isotopy, i.e. they are connected by a one-parameter family with constant intersection lattice, then their complements are diffeomorphic.  Based on Randell's theorem,  in \cite{Jiang1994} and \cite{Wang2005}, the authors found large classes of line arrangements, called \notion{nice arrangements} and \notion{simple arrangements} respectively, whose intersection lattices determine topology of the complements. The notion of nice line arrangements has been generalized to arrangements of hyperplanes in higher dimensional projective spaces (see \cite{Wang2007a, Wang2008, Yau2007, Yau2009}). Also based on Randell's theorem, Nazir-Yoshinaga \cite{Nazir2010} found new classes of line arrangements whose intersection lattices determine  the topology of the complements.  Unlike nice and simple arrangements whose intersection lattices have special properties, Nazir and Yashinaga's new classes require that all intersection points with multiplicity at least $3$ are in special positions.

However, even for line arrangements, the converse is not true in general.  We call a pair of line arrangements a \notion{Zariski pair}, if they are lattice isomorphic, but the fundamental groups of their complements are different. Note that this definition is stronger than the original definition introduced by Artal Bartolo in \cite{Artal-Bartolo1994}: a pair of lattice isomorphic line arrangements with different embedding type. The first Zariski pair of line arrangements was constructed by Rybnikov in 1998. But the work wasn't published until 2011  \cite{Rybnikov2011}. Each arrangement in Rybinikov's example consists of 13 lines and 15 triple points.   On the other hand,  Garber, Teicher and Vishne \cite{Garber2003} proved that there is no Zariski pair of arrangements of up to 8 real lines which covered the result of Fan \cite{Fan1997} on arrangements of 6 lines. This result was recently generalized to arrangements of 8 complex lines by Nazir and Yoshinaga
 \cite{Nazir2010}. In the same paper,  Nazir and Yoshinaga also claim without proof that there is no Zariski pair of arrangements of 9 complex lines. A complete proof of their claim was presented in \cite{Ye2011}.   For arrangements of 10 lines, it is still not known whether the fundamental groups of the complements are combinatorially  invariant. The existence of ``potential Zariski pairs" has been known (see the arrangements $\mathcal{H}^\pm$ in section 5 of \cite{ArtalBartolo2005} and Example 5.5 of \cite{Nazir2010}).

 Let  $\mathcal{A}$ be  a complex line arrangement.  We define the \notion{moduli space} of line arrangements with the fixed lattice $L(\mathcal{A})$ (or simply, the moduli space of $\mathcal{A}$) as \[\mathcal{M}_\mathcal{A}=\{\mathcal{B}\in ((\mathbb{CP}^2)^*)^n | L(\mathcal{B})= L(\mathcal{A})\}/ PGL(3, \mathbb{C}).\]  We denote by $\mathcal{M}^c_{\mathcal{A}}$ the quotient of $\mathcal{M}_{\mathcal{A}}$ under complex conjugation. We note that our moduli space $\mathcal{M}_{\mathcal{A}}$ is called an \notion{ordered} moduli spaces in \cite{ArtalBartolo2005}.
 By Randell's Lattice-Isotopy Theorem in \cite{Randell1989} and Cohen and Suciu's Theorem 3.9 in  \cite{Cohen1997}, we know that arrangements in the same connected component of the moduli space, or in two complex conjugate components can not form Zasiki pairs. Therefore, to investigate the existence of Zariski pairs of arrangements of 10 lines, it is very important to know the geometry of moduli spaces of arrangements.
In fact, results in our paper show that there are many arrangements of 10 lines whose moduli spaces are reducible.  

 The paper is structured as follows.  Section 2 provides preliminaries and ideas on classifying moduli spaces of arrangements of 10 lines. Section 3 shows that moduli spaces of arrangements with multiple points of high multiplicity are most likely irreducible. Section~ 4 and Section 5 deal with arrangements of 10 lines with a quadruple point. 
 All possible arrangements of 10 lines with quadruple point whose moduli spaces are reducible, for instance, consisting of 2 points, 3 points, 4 points, or 2 one-dimensional components, can be found there.

\paragraph{Acknowledgements:} This work was partially supported by the Oswald Veblen Fund and by the Minerva Foundation of Germany. The authors thank M. Falk and the referee for their comments and suggestions that helped improve the clarity of the paper.

\section{Preliminaries}

Let  $\mathcal{A}=\{L_1, L_2,\cdots, L_n\}$ be a line arrangement in $\mathbb{CP}^2$. We say a singularity of $L_1\cup
L_2\cup \cdots \cup L_n$ is a \notion{multiple point} of $\mathcal{A}$, if it has multiplicity at least 3.

\begin{Def}
A line arrangement $\mathcal{A}$ is said to be {\em $C_{\leq 3}$} if all multiple points of $\mathcal{A}$ are on at most
three lines, say $L_1$, $L_2$ and $L_3$. A  line arrangement is called {\em simple $C_{\leq 3}$} if it is $C_{\leq
3}$ and one of the following conditions holds:
\begin{enumerate}
\item  $L_1\cap L_2\cap L_3\neq \emptyset$, or \item $L_1\cap L_2\cap L_3=\emptyset$ and that one of $L_1$, $L_2$ and $L_3$ contains at most one
multiple point apart from $L_1\cup L_2\cup L_3$.
\end{enumerate}
\end{Def}

We recall the following results  of \cite{Nazir2010}.
\begin{thm}[Theorem 3.5, \cite{Nazir2010}]\label{thm:NY-Main}Let $\mathcal{A}$ be a simple $C_{\leq 3}$ arrangement.
 Then the moduli space $\mathcal{M}_\mathcal{A}$ is irreducible.
\end{thm}

\begin{thm}[Lemma 3.2, \cite{Nazir2010}]\label{lemma:extend-irreduciblity}
Let $\mathcal{A}=\{L_1, L_2, \cdots, L_n\}$ be a line arrangement.  Assume that $L_n$ passes through at most $2$
multiple points.   Set  $\mathcal{A}'=\{L_1, L_2, \cdots, L_{n-1}\}$, then $\mathcal{M}_\mathcal{A}$ is
irreducible if  $\mathcal{M}_{\mathcal{A}'}$ is irreducible.
\end{thm}

We say that a line arrangement is {\em non-reductive} if each line of the arrangement passes through at least 3 multiple points.

To classify geometric objects, the more invariants we know, the more sharp classification we can expect.
 For arrangements of lines,  numerical invariants which we will use include the highest multiplicity of multiple points, the number of multiple points of certain multiplicity, and special lines passing through given number and type of multiple points.
 The classification of moduli spaces consists of two steps. Firstly, we will roughly classify intersection lattices according to various numerical invariants. Secondly, we will write down defining equations involving parameters for a given intersection lattice. The space of parameters is the moduli space of the arrangement.

Denote by $n_r$ the number of intersection points of multiplicity $r$. We recall the following useful results.

\begin{lemma}[see for instance \cite{Hirzebruch1986}]\label{lemma:intersection-formula}
Let $\mathcal{A}$ be an arrangement of $k$ lines in $\mathbb{CP}^2$. Then
\[\frac{k(k-1)}{2}=\sum_{r\geq 2}\frac{r(r-1)n_r}{2}.\]
\end{lemma}

\begin{thm}[\cite{Hirzebruch1986}]\label{thm:Hirzebruch}Let $\mathcal{A}$ be an arrangement of $k$
lines in $\mathbb{CP}^2$. Assume that $n_{k}=n_{k-1}=n_{k-2}=0$. Then
\[n_2+\frac{3}{4}n_3\geq k +\sum\limits_{r\geq 5}(2r-9)n_r.\]
\end{thm}

\section{Arrangements of 10 lines with multiple points of multiplicity at least 5}
One can easily discover that arrangements with multiple points of high multiplicity will most likely have irreducible moduli spaces. Results in this section suggest that we should not expect Zariski pairs of arrangements of 10 lines with at least a quintuple point.

\begin{thm}\label{prop:geq-6}
Let $\mathcal{A}$ be an arrangement of $10$ lines. If there is a multiple point of multiplicity $\geq 6$, then
the moduli space $\mathcal{M}_\mathcal{A}$ is irreducible.
\end{thm}
\begin{proof}
Assume that $L_1\cap L_2\cap\cdots\cap L_6\neq \emptyset$. It is easy to check that at least one of the six
lines contains at most two multiple points. By Lemma \ref{lemma:extend-irreduciblity} and classification of arrangements of 9 lines (see \cite{Ye2011} Proposition 3.3), we see that $\mathcal{M}_\mathcal{A}$ is irreducible.
\end{proof}

When the highest multiplicity is 5, we notice that there is a case that the moduli space is reducible. However, a close look (see Remark \ref{rmk:extended-FS}) shows that the fundamental groups are still isomorphic.

\begin{thm}\label{prop:r=5}
Let  $\mathcal{A}$ be a non-reductive arrangement of $10$ lines with a quintuple point and no multiple points of higher multiplicities. Then $\mathcal{A}$ contains a Falk-Sturmfels arrangement as a sub-arrangement.
\end{thm}

\begin{proof}
By Lemma \ref{lemma:intersection-formula} and Theorem \ref{thm:Hirzebruch}, we have the following inequality $n_3+n_4\leq  \frac{140-44n_5}{9}$. Since $n_5\geq 1$, thus $n_3+n_4\leq 10$. On the other hand,  there must be at least $11-n_5$
multiple points so that each line will pass through at least 3 multiple points. Therefore,
$11-n_5\leq n_3+n_4$.
The two inequalities together tell us that  $n_5=1$ and $n_3+n_4=10$.  Apply Lemma \ref{lemma:intersection-formula} again, and we see that $n_4\leq 1$.

Assume that $L_1\cap L_2\cap L_3\cap L_4\cap L_5$ is the only quintuple point. There should be no multiple point apart from $L_1\cup L_2\cup \cdots \cup L_5$.  Otherwise,  $n_3+n_4\geq 11$, since each line passes through at least 3 multiple points. Moreover, each of the lines $L_1$, $L_2$, $\dots$, $L_5$ should pass through exactly 2 more multiple points.

Let $n_4= 1$. Then there are 9 triple points. We may assume that $L_1\cap L_6\cap L_7\cap L_8$ is the quadruple point. Since each line passes through 3 multiple points,  then $L_9\cap L_{10}$ must be a triple point on $L_1$. Now there are 3 multiple points on $L_1$. The rest of the multiple points are triple points which should be in $(L_9\cup L_{10})\cap (L_6\cup L_7\cup L_8)\cap (L_2\cup L_3\cup L_4\cup L_5)$. However, there are at most 6 triple points in the intersection.

Let $n_4=0$.  We may assume that  $L_1\cap L_6\cap L_7$ and $L_1\cap L_8\cap L_9$ are the 2 triple points on $L_1$. Then the sub-arrangement $\mathcal{A}':=\mathcal{A}\setminus\{L_1\}$ has 1 quadruple point, $L_2\cap L_3\cap L_4\cap L_5$, and 8 triple points. Note that the 8 triple points should be in $(L_6\cup L_7)\cap (L_8\cup L_9)$ and $L_{10}\cap (L_6\cup L_7\cup L_8\cup L_9)$. Therefore, each of $L_6$, $L_7$, $L_8$ and $L_9$ should pass through 3 triple points and $L_{10}$ should pass through 4 triple points.  By the classification of arrangements of 9 lines (see \cite{Ye2011} Proposition 3.4), $\mathcal{A}'$ is isomorphic to a Falk-Sturmfels arrangement.
\end{proof}

\begin{rmk}\label{rmk:extended-FS}
Note that the line arrangements  in Proposition \ref{prop:r=5} are lattice isomorphic to those extended Falk-Sturmfels arrangements defined in \cite{Ye2011}, Example 4.1. Hence the fundamental groups are isomorphic.
\end{rmk}

A surprising corollary of Theorem \ref{prop:geq-6} and \ref{prop:r=5} is that there is no Zariski pair of arrangements of 10 lines with a multiple point whose multiplicity is at least 5.

\begin{thm}
Let $\mathcal{A}$ be an arrangements of $10$ lines such that $n_{r}\geq 1$ for some $r\geq 5$. Then the fundamental groups of the complement $\pi_{1}(M(\mathcal{A}))$ is determined by the intersection lattice $L(\mathcal{A})$.
\end{thm}
\begin{proof}
It suffices to consider cases that $\mathcal{A}$ has a quintuple point and contains a MacLane arrangement or a Falk Sturmfels arrangement. However, it is not hard to check that  adding two lines to a MacLane arrangement does not change the irreducibility of $\mathcal{M}^{c}_{\mathcal{A}}$. In fact, the two lines added to the MacLane arrangements must pass through a triple point and at most one triple point. Now, assume that $\mathcal{A}$ contains a Falk-Sturmfels arrangement. In other words,  $\mathcal{A}$ can be obtained from a Falk-Sturmfels arrangement by adding a line passing through the unique quadruple point of the Falk-Sturmfels arrangement. Note that a Falk-Sturmfels arrangement has the following feature: all but 2 of the intersection points in the intersection lattice are on the union of the lines passing through the quadruple point.  Moreover, the 2 double points and the quadruple point are collinear, say on the line $H$.  Recall that there is an automorphism $\varphi$ of $\mathbb{CP}^{2}$ which sends one of the Falk-Sturmfels arrangement to the other one and fixes the line $H$. Let $\mathcal{A}_{1}$ and $\mathcal{A}_{2}$ be two lattice isomorphism arrangements of 10 lines which contain Falk-Sturmfels arrangements. Then it is easy to see that either they are in the same irreducible component, or the complements $M(\mathcal{A}_{1})$ and $M(\mathcal{A}_{2})$ are diffeomorphic under the automorphism $\varphi$.
 \end{proof}

\section{Arrangements of 10 lines with 2 or more quadruple points}
In this section, we investigate an arrangement of 10 lines with at least 2 quadruple points and no multiple points of higher multiplicities.  First, let us consider possible values of the numerical invariant $n_4$ such that the arrangement is non-reductive.

\begin{lemma}\label{lemma:atmost3quadruple} Let $\mathcal{A}$ be a non-reductive arrangement of 10 lines in $\mathbb{CP}^2$ with $n_r=0$ for $r\geq 5$. Then $n_4 \leq 3$.
\end{lemma}
\begin{proof}
By Theorem \ref{thm:Hirzebruch} and Lemma
\ref{lemma:intersection-formula}, we know that $n_3\leq \frac{140-24n_4}{9}$. Since each line passes through at least 3
multiple points, then there must be at least $9-n_4$ triple points. Therefore, we obtain that $n_4\leq
3$ from the inequality $9-n_4\leq n_3\leq  \frac{140-24n_4}{9}$.
\end{proof}


Notice that the number of quadruple points is still a very rough combinatorial invariant. Fixing this invariant, we want more combinatorial invariants. An obvious one is the collinearity of those quadruple points in the arrangement.  Another invariant that we frequently used is the numbers of multiple points on lines.  By investigating possible multiple points on lines, we are able to narrow down the number of classes of arrangements with desired combinatorial properties and then write down definition equations without difficulties.

\subsection{$n_4=3$}
We first consider that arrangements have 3 quadruple points. By Lemma \ref{lemma:intersection-formula}, there are at most 7 triple points.

\begin{thm}\label{prop:r=4}Let $\mathcal{A}$ be a non-reductive arrangement of $10$ lines in $\mathbb{CP}^2$ such that
 $n_4=3$ and $n_r=0$ for $r\geq 5$. Then the moduli space $\mathcal{M}_{\mathcal{A}}$ or  $\mathcal{M}^c_{\mathcal{A}}$ is irreducible.
 \end{thm}
\begin{proof}Given 2 non-collinear quadruple points in $\mathcal{A}$, we claim that there will be a line passing through only 2 multiple points. To see that, we let $L_1\cap L_2\cap L_3\cap
L_4$ and $L_5\cap L_6\cap L_7\cap L_8$ be the 2 quadruple points. Then the third quadruple point must be $L_i\cap L_j \cap L_9\cap L_{10}$ for some $i\in\{1, 2, 3 ,4\}$ and  $j\in \{5, 6, 7, 8\}$. Then there are only 2 multiple points (in fact, 2 quadruple points) on $L_i$, as well as on $L_j$.

Consider that the 3 quadruple points are not collinear, but any 2 of them are collinear in $\mathcal{A}$.  Let $L_1\cap L_2\cap L_3\cap L_{10}$, $L_4\cap L_5\cap L_6\cap L_{10}$ and $L_3\cap L_4\cap L_7\cap L_8$ be the 3 quadruple points.  Since each line passes through 3 multiple points, there must be a triple point on $L_{10}$. We may assume that $L_7\cap L_9\cap L_{10}$ is the triple point.

Assume that $L_8\cap L_9$ is not a triple point.  Then $L_8$ should contain either $\{L_1\cap L_5, L_2\cap L_6\}$ (see Figure \ref{n4=3-1}) or $\{L_1\cap L_6, L_2\cap L_5\}$  so that there will be 3 multiple points on $L_8$. Up to a lattice-isomorphism, we may assume that $L_8$ contains $\{L_1\cap L_5, L_2\cap L_6\}$. Similarly,  we assume that $L_9$ passes through $L_1\cap L_4$ and $L_3\cap L_5$ or $L_1\cap L_4$ and $L_3\cap L_6$ so that each of  $L_3$, $L_4$ and $L_9$ will pass through 3 multiple points. If $L_9$ passes through $L_1\cap L_4$ and $L_3\cap L_5$, then $L_7$ must pass through $L_1\cap L_6$ and $L_2\cap L_5$ so that each of  $L_2$ and $L_6$ passes through $3$ multiple points.  Consequently,  the arrangement is  lattice isomorphic to the arrangement in Figure \ref{n4=3-1} merged with $L_{10}$ at infinity.
One can check that the moduli space is irreducible (cf. calculations in proofs of forthcoming theorems). If $L_9$ passes through $L_1\cap L_4$ and $L_3\cap L_6$, then $L_7$ must pass through $L_2\cap L_5$ but not $L_1\cap L_6$ (otherwise the arrangement cannot be realized). In this case, we notice that the dual arrangement (in the sense that multiple points go to lines and lines go to points) consists of $9$ lines and $10$ triple points such that each line passes through at least $3$ triple points. By \cite{Ye2011} Proposition 3.8, it is lattice isomorphic to the $\mathcal{A}^{\pm\sqrt{-1}}$ arrangement (see \cite{Ye2011} Example 2.3). Hence the moduli space $\mathcal{M}^c_{\mathcal{A}}$ is irreducible.
\begin{figure}[htbp]
\begin{minipage}[b]{0.5\linewidth}
\centering
\begin{tikzpicture}[domain=-2:4, scale=0.6]
\tikzstyle{every node}=[font=\footnotesize]
\draw[domain=-2:4] plot (0, \x) node[above]{$L_5$};
\draw[domain=-2:4] plot (1, \x) node[above]{$L_4$};
\draw[domain=-2:4] plot (2, \x) node[above]{$L_6$};

\draw[domain=4:-2] plot (\x, 0) node[left]{$L_2$};
\draw[domain=4:-2] plot (\x, 1) node[left]{$L_3$};
\draw[domain=4:-2] plot (\x, 2) node[left]{$L_1$};

\draw[domain={-2:4}] plot (\x, \x) node[above]{$L_{7}$};
\draw[domain={4:-2}] plot (\x, 2-\x) node[above]{$L_{8}$};
\draw[domain=-2:3] plot (\x, {\x+1}) node[above]{$L_9$};
\end{tikzpicture}
\caption{\label{n4=3-1}}
\end{minipage}
\begin{minipage}[b]{0.5\linewidth}
\centering
\begin{tikzpicture}[scale=0.6]
\tikzstyle{every node}=[font=\footnotesize]
\draw[domain=-2.5:3.5] plot (0, \x) node[above]{$L_5$};
\draw[domain=-2.5:3.5] plot (1, \x) node[above]{$L_4$};
\draw[domain=-2.5:3.5] plot ({(sqrt(5)+3)/2}, \x) node[above]{$L_6$};

\draw[domain=3.5:-2.5] plot (\x, 0) node[left]{$L_2$};
\draw[domain=3.5:-2.5] plot (\x, 1) node[left]{$L_3$};
\draw[domain=3.5:-2.5] plot (\x, {(sqrt(5)+3)/2}) node[left]{$L_1$};

\draw[domain={-2.5:3.5}] plot (\x, \x) node[above]{$L_{7}$};
\draw[domain={3.5:-2.5}] plot (\x, {-2/(1+sqrt(5))*(\x-(3+sqrt(5))/2)}) node[left]{$L_{8}$};
\draw[domain=3:-2.5] plot (\x, {-\x+1}) node[above]{$L_9$};
\end{tikzpicture}
\caption{\label{n4=3-2}}
\end{minipage}
\end{figure}

Assume that $L_8\cap L_9$ is also a triple point  (see Figure \ref{n4=3-2}). We may assume that  $L_8\cap L_9$ is on $L_1$.  First, we note that $L_3\cap L_9$ and $L_2\cap L_4\cap L_9$ must be triple points so that each of $L_3$ and $L_4$ will pass through 3 multiple points. We may assume that $L_3\cap L_9$ is on $L_5$. Then $L_2\cap L_6\cap L_8$ and $L_1\cap L_6\cap L_7$ must be triple points so that $L_6$ will pass through 3 multiple points. Now consider  $L_5$. We see that $L_2\cap L_5\cap L_7$ must be a triple point so that $L_5$ will pass through 3 multiple points. Let $L_{10}$ be the line at infinity.  It is easy to check that $L_7$ can not be parallel to $L_9$, i.e., $L_7\cap L_9$ can not be on the line at infinity $L_{10}$. We have a contradiction.

Assume that the 3 quadruple points are collinear in $\mathcal{A}$ (see Figure
\ref{fig:n4-A}.). We may assume that $L_1\cap L_2\cap L_3\cap L_{10}$, $L_4\cap L_5\cap
L_6\cap L_{10}$ and $L_7\cap L_8\cap L_9\cap L_{10}$ are the quadruple points.  Consider the sub-arrangment $\mathcal{A}':=\mathcal{A}\setminus\{L_{10}\}$. It has at most 10 triple points and no multiple points of higher multiplicities. Moreover, we note that each line of $\mathcal{A}'$ passes through at least  3 triple points and the 3 triple points $L_1\cap L_2\cap L_3$, $L_4\cap L_5\cap
L_6$, $L_7\cap L_8\cap L_9$ are collinear.  It is not difficult to check that $\mathcal{A'}$ is lattice isomorphic to one of the arrangements in Figure
\ref{fig:n4-A}.
\begin{figure}[htbp]
\centering \subfigure{\begin{tikzpicture}[domain=-2:4, scale=0.6] \tikzstyle{every node}=[font=\footnotesize]
\draw[domain=-2:4] plot (0, \x) node[above]{$L_4$};
\draw[domain=-2:4] plot (1, \x) node[above]{$L_5$};
\draw[domain=-2:4] plot (2, \x) node[above]{$L_6$};

\draw[domain=4:-2] plot (\x, 0) node[left]{$L_1$};
\draw[domain=4:-2] plot (\x, 1) node[left]{$L_2$};
\draw[domain=4:-2] plot (\x, 2) node[left]{$L_3$};

\draw[domain={-1.5:3.5}] plot (\x, \x) node[right]{$L_{8}$};
\draw[domain=-2:3] plot (\x, {\x+1}) node[right]{$L_7$};
\draw[domain=-1:4] plot (\x, {(\x-1)}) node[right]{$L_9$};
\end{tikzpicture}}
\hspace{2cm}
\subfigure{\begin{tikzpicture}[domain=-1.5:4.5,scale=0.6] \tikzstyle{every node}=[font=\footnotesize]
\draw[domain=-1.5:3.5] plot (0, \x) node[above]{$L_4$};
\draw[domain=-1.5:3.5] plot (1, \x) node[above]{$L_5$};
\draw[domain=-1.5:3.5] plot (3, \x) node[above]{$L_6$};

\draw[domain=4.5:-1.5] plot (\x, 0) node[left]{$L_1$};
\draw[domain=4.5:-1.5] plot (\x, 1) node[left]{$L_2$};
\draw[domain=4.5:-1.5] plot (\x, 3/2) node[left]{$L_3$};

\draw plot (\x, {\x/2}) node[right]{$L_7$};
\draw plot (\x, {\x/2+1}) node[right]{$L_8$};
\draw plot (\x, {(\x-1)/2}) node[right]{$L_9$};
\end{tikzpicture}} \caption{\label{fig:n4-A}}
\end{figure}
Hence the moduli space $\mathcal{M}_{\mathcal{A}}$ is irreducible.
\end{proof}

\subsection{$n_4=2$}
Since $n_4=2$, then there are at most 9 triple points.  We first consider that the 2 quadruple points are not collinear.

\begin{thm}Let $\mathcal{A}$ be a non-reductive arrangement of $10$ lines such that $n_4=2$ and $n_r=0$ for $r\geq 5$. If the $2$ quadruple points are not on the same line of $\mathcal{A}$, then the quotient moduli space $\mathcal{M}^c_{\mathcal{A}}$ is irreducible.
\end{thm}
\proof
We may assume that $L_1\cap L_2\cap L_3\cap L_4$ and $L_5\cap L_6\cap L_7\cap L_8$ are the two quadruple points. 
Then each of $L_9$ and $L_{10}$ must pass through 4 of the points in $(L_1\cup L_2\cup L_3\cup L_4)\cap (L_5\cup L_6\cup L_7\cup L_8)$ so that each line will pass through at least 3 multiple points. By lattice isomorphism, we may assume that  $L_i\cap L_{9-i}\cap L_9$, where $i=1,2,3,4$, and $L_4\cap L_8\cap L_{10}$ are triple points (see Figure \ref{n4=2-1}).
\begin{figure}[htbp]
\centering
\begin{tikzpicture}[domain=-2:4,scale=0.6]
\tikzstyle{every node}=[font=\footnotesize]
\draw[domain=-1:4] plot (0, \x) node[above]{$L_5$};
\draw[domain=-1:4] plot (1, \x) node[above]{$L_6$};
\draw[domain=-1:4] plot (2, \x) node[above]{$L_7$};
\draw[domain=-1:4] plot (3, \x) node[above]{$L_8$};

\draw[domain=4:-1] plot (\x, 0) node[left]{$L_4$};
\draw[domain=4:-1] plot (\x, 1) node[left]{$L_3$};
\draw[domain=4:-1] plot (\x, 2) node[left]{$L_2$};
\draw[domain=4:-1] plot (\x, 3) node[left]{$L_1$};

\draw[domain={-1:4}] plot (\x, \x) node[above]{$L_{9}$};
\draw[domain={2.5:3.5}] plot (\x, 3-\x) node[below right]{$L_{10}$};
\end{tikzpicture}
\caption{\label{n4=2-1}}
\end{figure}
 Then there are only 2 lattice isomorphism classes of arrangements satisfying our assumption. They are determined by triple points on $L_{10}$. The possible sets of triple points on $L_{10}$ are  $\{L_1\cap L_5, L_2 \cap L_6, L_3\cap L_7, L_4\cap L_8\}$, $\{L_1\cap L_7, L_2\cap L_6, L_3\cap L_5, L_4\cap L_8\}$, and $\{L_1\cap L_6, L_2 \cap L_5, L_3\cap L_7, L_4\cap L_8\}$.  However, by a permutation $(2, 3)(6, 7)$, the last two sets determine the same arrangement.

Assume that  $L_1=\{x=t_2z\}$, $L_2=\{x=t_1z\}$, $L_3=\{x=z\}$, $L_4=\{x=0\}$, $L_5=\{y=0\}$, $L_6=\{y=z\}$,  $L_7=\{y=t_1z\}$, $L_8=\{y=t_2z\}$ and $L_9=\{x=y\}$ where $t_1, t_2\in\mathbb{C}\setminus\{0 ,1\}$ and $t_1\neq t_2$.
If $L_{10}$ contains the set of points $\{L_i\cap L_{i+4} \mid i=1, 2, 3, 4\}$, then we get $t_2=t_1+1$ and $L_{10}=\{x+y=(t_1+1)z\}$. Hence the moduli space $\mathcal{M}_{\mathcal{A}}$ is irreducible.
If $L_{10}$ contains the set $\{L_1\cap L_7, L_2\cap L_6, L_3\cap L_5, L_4\cap L_8\}$, then $t_1=1+(\pm\sqrt{-1})$, $t_2=\pm\sqrt{-1}$, and $L_{10}=\{y=(t_1-1)x+z\}$. Therefore $\mathcal{M}^c_{\mathcal{A}}$ is irreducible.

\qed

Now we consider the case that the 2 quadruple points are collinear in $\mathcal{A}$.

\begin{thm}
Let $\mathcal{A}$ be a non-reductive arrangement of $10$ lines such that $n_4=2$ and $n_r=0$ for $r\geq 5$. If the $2$ quadruple points are on the same line of $\mathcal{A}$, then either the quotient moduli space $\mathcal{M}^c(\mathcal{A})$ is irreducible, or $\mathcal{A}$ is isomorphic to one of the arrangements defined by equations \eqref{n_4=2-3}, \eqref{n_4=2-5}, \eqref{n_4=2-4} and \eqref{n_4=2-7}.
\end{thm}
\begin{proof}
We may assume that $L_1\cap L_2\cap L_3\cap L_{10}$ and $L_4\cap L_5\cap L_6\cap L_{10}$ are the two quadruple points and $L_7\cap L_8\cap L_{10}$ is the only triple point on $L_{10}$. Then all triple points on $L_9$ are on $L_1\cup L_2\cup \cdots L_8$.
\paragraph{Case 1:}
Assume that none of $L_9\cap (L_7\cup L_8)$ is a triple point of $\mathcal{A}$. Then $L_9$ should pass through exactly 3 triple points in $(L_1\cup L_2\cup\cdots L_6)\setminus(L_7\cup L_8\cup L_{10})$. We may assume that the 3 triple points on $L_9$ are $L_i\cap L_{7-i}$, where $i\in\{1, 2, 3\}$. Then each of $L_7$ and $L_8$ must pass through at least 2 of $\{L_i\cap L_{3+j}\mid i, j\in\{1, 2, 3\}\}$ and $L_8\cap L_9$ is on $L_{10}$.  Moreover, since there are at most 9 triple points, then only one of $L_7$ and $L_8$ can pass through 3 of  $\{L_i\cap L_{3+j}\mid i, j\in\{1, 2, 3\}\}$. We may assume that $L_7$ passes through 2 of  $\{L_i\cap L_{3+j}\mid i, j\in\{1, 2, 3\}\}$.
By a  lattice isomorphism, we may assume that $L_2\cap L_4$ is on $L_7$.  By an automorphism of the dual of the projective plane, we may assume that $L_{10}=\{z=0\}$, $L_1=\{x=tz\}$, $L_2=\{x=z\}$, $L_3=\{x=0\}$, $L_4=\{y=0\}$, $L_5=\{y=z\}$, $L_6=\{y=tz\}$, and $L_{9}=\{x=y\}$ where $t$ is a complex number such that $t\neq 0, 1$ (see Figure \ref{n4=2-2}).
\begin{figure}[htbp]
\begin{minipage}[b]{0.5\linewidth}
\centering
\begin{tikzpicture}[domain=-2:4,scale=0.8]
\tikzstyle{every node}=[font=\footnotesize]
\draw[domain=-1:3] plot (0, \x) node[above]{$L_4$};
\draw[domain=-1:3] plot (1, \x) node[above]{$L_5$};
\draw[domain=-1:3] plot (2, \x) node[above]{$L_6$};

\draw[domain=3:-1] plot (\x, 0) node[left]{$L_3$};
\draw[domain=3:-1] plot (\x, 1) node[left]{$L_2$};
\draw[domain=3:-1] plot (\x, 2) node[left]{$L_1$};

\draw[domain={-1:3}] plot (\x, \x) node[above]{$L_{9}$};
\draw[domain={0.2:-0.2}] plot (\x, \x+1) node[below left]{$L_{7}$};
\end{tikzpicture}
\caption{\label{n4=2-2}}
\end{minipage}
 \begin{minipage}[b]{0.5\linewidth}
 \centering
\begin{tikzpicture}[domain=-2:4,scale=0.7]
\tikzstyle{every node}=[font=\footnotesize]
\draw[domain=-1:3.5] plot (0, \x) node[above]{$L_4$};
\draw[domain=-1:3.5] plot (1, \x) node[above]{$L_5$};
\draw[domain=-1:3.5] plot (2, \x) node[above]{$L_6$};

\draw[domain=3:-1] plot (\x, 0) node[left]{$L_3$};
\draw[domain=3:-1] plot (\x, 1) node[left]{$L_2$};
\draw[domain=3:-1] plot (\x, 2) node[left]{$L_1$};

\draw[domain={2.2:-0.6}] plot (\x, 1.5*\x) node[left]{$L_{9}$};
\draw[domain={0.5:0.8}] plot (\x, -\x+5/3);
\node at (0.4,1.5){$L_{8}$};
\draw[domain={1.2:0.8}] plot (\x, -0.5*\x+2);
\node at (1.5, 1.3) {$L_{7}$};
\end{tikzpicture}
\caption{\label{n4=2-3-6}}
\end{minipage}
\end{figure}

\begin{enumerate}
\item
Assume that $L_1\cap L_5$ is on $L_7$, then $L_7=\{y=(t-1)x+z\}$.  Since the intersection point $L_7\cap L_8$ is on $L_{10}$, then the equation of the line $L_8$ has the form $y=(t-1)x+cz$, where $c$ is a complex number and $c\neq 1$.
Since each line passes through at least 3 multiple points, then either $\{L_3\cap L_6, L_1\cap L_4\}\subset L_8$ or $\{L_3\cap L_5, L_2\cap L_6\}\subset L_8$.
If $\{L_3\cap L_6, L_1\cap L_4\}\subset L_8$, then $t=0$. We get a contradiction..
If $\{L_3\cap L_5, L_2\cap L_6\}\subset L_8$, then we have $t=2$ and $L_8=\{y=x-z\}$. It follows that $L_7\cap L_8\cap L_9\cap L_{10}$ is an extra quadruple points.  Again, this contradicts our assumption.
\item
Assume that $L_3\cap L_5$ is  in $L_7$, then $L_{7}=\{x+y=z\}$ and  $L_8=\{x+y=cz\}$, where $c\neq 1$ is a complex number. Then $L_8$ passes through exactly 2 of $\{L_1\cap L_4, L_1\cap L_5, L_2\cap L_6, L_3\cap L_6\}$. Note that the coordinates of the points in $\{L_1\cap L_4, L_1\cap L_5, L_2\cap L_6, L_3\cap L_6\}$ are of the form $(\alpha(t), \beta(t), 1)$, where $\alpha(t)$ and $\beta(t)$ are linear functions of $t$. Then we obtain a system of  two linear equations in $t$ and $c$. Thus the moduli space is either empty or irreducible.
\item
Assume that  $L_3\cap L_6$ is  in $L_7$, then $L_7=\{y=-\frac{1}{t}x+z\}$ and $L_8=\{y=-\frac{1}{t}x+cz\}$, where $c$ is a complex number such that $c\neq 1$. Since each line has at least 3 multiple points, then $L_8$ must pass through one of $L_1\cap L_4$ and $L_1\cap L_5$. If $L_1\cap L_4$ is on $L_8$, then $L_3\cap L_5$ must be on $L_8$ so that $L_5$ will pass through 3 multiple points. Then $t=-1$ which implies that $L_7\cap L_8\cap L_9\cap L_{10}$ is an extra quadruple points. This is a contradiction to our assumption.  If $L_1\cap L_5$ is on $L_8$, then $L_2\cap L_6$ should be on $L_8$ so that $L_{8}$ passes through 3 multiple points.  Then $t=1$. This is a contradiction again.
\end{enumerate}
\paragraph{Case 2:}
Assume that both $L_7\cap L_9$ and $L_8\cap L_9$ are triple points. Then at least one of $L_9\cap L_i\cap L_j$, where $i\in\{1, 2,3\}$ and $j\in\{4, 5, 6\}$, is a triple point so that $L_9$ will pass through 3 multiple points. Assume that $L_9\cap L_3\cap L_4$ is a triple point. Moreover, we may assume that $L_7\cap L_9\cap L_5$ is a triple point.
Consider the intersection of $L_1$, $L_6$ and $L_9$. There are two possibilities.
\paragraph{(1)}Consider that $L_1\cap L_6\cap L_9=\emptyset$. Then $L_2\cap L_8$ or $L_1\cap L_8$ should be in $L_9$ so that each line passes through 3 multiple points. We may assume that $L_2\cap L_8\cap L_9$ is a triple point   (see Figure \ref{n4=2-3-6}).   Consequently,  either $L_1\cap L_6$ or $L_3\cap L_6$ should be on $L_8$, otherwise $L_6$ will pass through at most 2 multiple points.
\begin{enumerate}
\item
Assume that $L_1\cap L_6$ is on $L_8$. Then $L_3\cap L_5$ should be on $L_8$ so that $L_5$ will pass through 3 multiple points. It is easy to see that $L_7$ must pass through $L_1\cap L_4$ and $L_2\cap L_6$ so that each of  $L_1$ and $L_2$ will pass through 3 multiple points.  Recall that $L_7\cap L_8$ is on $L_{10}$.  Then equations of such arrangements can be written as:
\begin{equation*}\label{n_4=2-2}\scriptstyle  xyz(x-z)(y-z)(x-t z)(y-t z)(y-\frac{t}{t-1} (x-z))(y-\frac{t}{t-1}(x+t z))(y-\frac{t^2}{t-1}x)=0\end{equation*}
where $t=\frac{1\pm\sqrt{-1}}{2}$.
\item
Assume that $L_3\cap L_6$ is on $L_8$. Then $L_1\cap L_5$ should be on $L_8$ so that $L_5$ will pass through 3 multiple points.  Consequently, we should have $\{L_1\cap L_4, L_2\cap L_6\}\subset L_7$ or $\{L_2\cap L_4, L_1\cap L_6\}\subset L_7$.
Assume that $\{L_1\cap L_4, L_2\cap L_6\}$ is in $L_7$.  Then
the equations of the arrangements can be written as:
\begin{equation}\label{n_4=2-3}\scriptstyle
xyz(x-z)(y-z)(x-t z)(y-(1-t) z)(y-x-(1-t)z)(y-x+t z)(y-(2-t)x)=0 
\end{equation}
where $t=\frac{1\pm\sqrt{5}}{2}$.
Clearly, they are real arrangements (affine pictures are shown in Figure \ref{n4=2-real-excptional-1}).
\begin{figure}[htbp]
\centering
\subfigure{}{
\begin{tikzpicture}[domain=-2:4,scale=0.8]
\tikzstyle{every node}=[font=\footnotesize]
\draw[domain=-2:2] plot (0, \x) node[above]{$L_4$};
\draw[domain=-2:2] plot (1, \x) node[above]{$L_5$};
\draw[domain=-2:2] plot (({(1+sqrt(5))/2}, \x) node[above]{$L_6$};

\draw[domain=3:-1] plot (\x, 0) node[left]{$L_3$};
\draw[domain=3:-1] plot (\x, 1) node[left]{$L_2$};
\draw[domain=3:-1] plot (\x, {(1-sqrt(5))/2}) node[left]{$L_1$};

\draw[domain={-1:3}] plot (\x, {((3-sqrt(5))/2)*\x}) node[right]{$L_{9}$};
\draw[domain={2.5:-1}] plot (\x, {\x+(1-sqrt(5))/2}) node[left]{$L_7$};
\draw[domain={3:-0.2}] plot (\x, {\x-(1+sqrt(5))/2}) node[left]{$L_8$};
\end{tikzpicture}
}
\subfigure{}{
\begin{tikzpicture}[domain=-2:4,scale=0.9]
\tikzstyle{every node}=[font=\footnotesize]
\draw[domain=-1:3] plot (0, \x) node[above]{$L_4$};
\draw[domain=-1:3] plot (1, \x) node[above]{$L_5$};
\draw[domain=-1:3] plot (({(1-sqrt(5))/2}, \x) node[above]{$L_6$};

\draw[domain=2.5:-1.5] plot (\x, 0) node[left]{$L_3$};
\draw[domain=2.5:-1.5] plot (\x, 1) node[left]{$L_2$};
\draw[domain=2.5:-1.5] plot (\x, {(1+sqrt(5))/2}) node[left]{$L_1$};

\draw[domain={1.15:-0.2}] plot (\x, {((3+sqrt(5))/2)*\x}) node[below]{$L_{9}$};
\draw[domain={1.3:-1}] plot (\x, {\x+(1+sqrt(5))/2}) node[left]{$L_7$};
\draw[domain={2.3:-1}] plot (\x, {\x-(1-sqrt(5))/2}) node[left]{$L_8$};
\end{tikzpicture}
}
\caption{\label{n4=2-real-excptional-1}}
\end{figure}
Assume that $\{L_2\cap L_4, L_1\cap L_6\}$ is in $L_7$. By the permutation $(1,6)(2,5)(3,4)(7,8)$,  we know that the arrangement is isomorphic to arrangements defined by equation \eqref{n_4=2-2}.
\end{enumerate}
\paragraph{(2)}Consider that $L_1\cap L_6\cap L_9\neq\emptyset$. Then $L_2\cap L_8$ should be on $L_9$ by assumption. Moreover, $L_8$ should pass through either $L_1\cap L_5$ or $L_3\cap L_5$ so that $L_5$ will pass through 3 multiple points. We may assume that $L_8$ passes through $L_3\cap L_5$ (see Figure \ref{n4=2-3}).
\begin{figure}[htbp]
 \begin{minipage}[b]{0.5\linewidth}
 \centering
\begin{tikzpicture}[domain=-2:4,scale=0.7]
\tikzstyle{every node}=[font=\footnotesize]
\draw[domain=-1:3.5] plot (0, \x) node[above]{$L_4$};
\draw[domain=-1:3.5] plot (1, \x) node[above]{$L_5$};
\draw[domain=-1:3.5] plot (2, \x) node[above]{$L_6$};

\draw[domain=3:-1] plot (\x, 0) node[left]{$L_3$};
\draw[domain=3:-1] plot (\x, 1) node[left]{$L_2$};
\draw[domain=3:-1] plot (\x, 3) node[left]{$L_1$};

\draw[domain={2.2:-0.6}] plot (\x, 1.5*\x) node[left]{$L_{9}$};
\draw[domain={0.5:1.167}] plot (\x, -3*\x+3);
\node at (1.5,-0.5){$L_{8}$};
\draw[domain={1.2:0.8}] plot (\x, -0.5*\x+2);
\node at (1.52, 1.33) {$L_{7}$};
\end{tikzpicture}
\caption{\label{n4=2-3}}
\end{minipage}
\begin{minipage}[b]{0.5\linewidth}
\centering
\begin{tikzpicture}[domain=-2:4,scale=0.7]
\tikzstyle{every node}=[font=\footnotesize]
\draw[domain=-1:3.5] plot (0, \x) node[above]{$L_4$};
\draw[domain=-1:3.5] plot (1, \x) node[above]{$L_5$};
\draw[domain=-1:3.5] plot (2.5, \x) node[above]{$L_6$};

\draw[domain=3.5:-1] plot (\x, 0) node[left]{$L_3$};
\draw[domain=3.5:-1] plot (\x, 1) node[left]{$L_2$};
\draw[domain=3.5:-1] plot (\x, 2) node[left]{$L_1$};

\draw[domain={-1:3.5}] plot (\x, \x) node[above]{$L_{9}$};
\draw[domain={1.5:3.2}] plot (\x, -2*\x+6) node[below]{$L_7$};
\end{tikzpicture}
\caption{\label{n4=2-4}}
\end{minipage}
\end{figure}
Now look at $L_1$. One of $L_7$ and $L_8$ must pass through $L_1\cap L_4$ so that $L_1$ passes through 3 multiple points.
\begin{enumerate}
\item Assume that $L_7$ passes through $L_1\cap L_4$. Then $L_7$ must also pass through $L_2\cap L_6$ so that each of $L_2$ and $L_6$ passes through 3 multiple points. Then equations of arrangements with such an intersection lattice can be written as: \begin{equation} \label{n_4=2-5}\scriptstyle
xyz(x-z)(y-z)(x-t_1 z)(y-t_2 z)(y-x)(y-(t_2-1)x - z)(y-(t_2-1)x-(t_1-1)z)=0
 \end{equation}
where $t_1=(t_2-1)^2$ and $(t_2-1)^3-(t_2-1)^2+1=0$. Therefore there are 1 real and 2 complex arrangements in the moduli space.  The real arrangement is shown in Figure \ref{n4=1-real-2-complex}.
 \begin{figure}[htbp]
\centering
\begin{tikzpicture}[domain=-2:4,scale=1]
\tikzstyle{every node}=[font=\footnotesize]
\draw[domain=-1:1.8] plot (0, \x) node[above]{$L_4$};
\draw[domain=-1:1.8] plot (1, \x) node[above]{$L_6$};
\draw[domain=-1:1.8] plot ({0.5698402909980533}, \x) node[above]{$L_5$};

\draw[domain=2:-1] plot (\x, 0) node[left]{$L_3$};
\draw[domain=2:-1] plot (\x, 1) node[left]{$L_1$};
\draw[domain=2:-1] plot (\x, {0.24512233375330719}) node[left]{$L_2$};

\draw[domain={-1:1.8}] plot (\x, \x) node[above]{$L_{9}$};
\draw[domain={2:-1}] plot (\x, {-0.7548776662466928*\x+1}) node[above]{$L_7$};
\draw[domain={1.8:-1}] plot (\x, {-0.7548776662466928*\x+0.4301597090019467}) node[above left]{$L_8$};
\end{tikzpicture}
\caption{\label{n4=1-real-2-complex}}
\end{figure}
\item Assume that $L_8$ passes through $L_1\cap L_4$. Then $L_2\cap L_7$ and $L_6\cap L_7$ must be triple points so that each of $L_2$ and $L_6$ passes through 3 multiple points.  If the 2 points $L_2\cap L_7$ and $L_6\cap L_7$ coincide, i.e., $L_2\cap L_6\cap L_7$ is a triple point, then the arrangement isomorphic to the one in previous case by a permutation (1, 4)(2, 5)(3, 6)(7, 8). Assume that $L_2\cap L_6\cap L_7$ is empty. Then $L_2\cap L_4\cap L_7$ and $L_3\cap L_6\cap L_7$ should be triple points. It is not hard to check that the moduli space is empty.
\end{enumerate}
\paragraph{Case 3:}
Assume that only one of $L_9\cap L_7$ and $L_9\cap L_8$ is a triple point. We may assume that $L_1\cap L_7\cap L_9$ is a triple point.  Then there will be 2 more triple points on $L_9$. Those 2 triple points are in $\{L_i\cap L_{3+j}\mid i, j\in\{2, 3\}$. Therefore, we may assume that the 3 triple points on $L_9$ are $L_3\cap L_4\cap L_9$, $L_2\cap L_5\cap L_9$, and $L_1\cap L_7\cap L_9$ (see Figure \ref{n4=2-4}). Consequently,  either $L_2\cap L_6$ or $L_3\cap L_6$ has to be on $L_7$ so that $L_6$ will pass through 3 multiple points.  Note that up to a permutation $(2, 3)(4, 5)$ we may assume that $L_2\cap L_6$ is on $L_7$. By an automorphism of the dual of the projective plane, we can write $L_{10}=\{z=0\}$, $L_1=\{x=t_1z\}$, $L_2=\{x=z\}$, $L_3=\{x=0\}$, $L_4=\{y=0\}$, $L_5=\{y=z\}$, $L_6=\{y=t_2z\}$, $L_9=\{x=y\}$, $L_7=\{y=\frac{t_2-1}{t_2-t_1}(x-t_1z)+z\}$ ,and $L_8=\{y=\frac{t_2-1}{t_2-t_1}x+cz\}$, where $t_1, t_2\in \mathbb{C}\setminus\{0, 1\}$, $t_1\neq t_2$,  and $c\neq \frac{t_1(t_2-1)}{t_2-t_1}+1\in\mathbb{C}$.

By our assumption, $L_8$ must pass through one of $\{L_1\cap L_4, L_1\cap L_5, L_1\cap L_6\}$ so that $L_1$ will pass through 3 multiple points. 

Let $L_1\cap L_4$ be on $L_8$. Then $L_8$ must pass through $L_3\cap L_6$ so that $L_6$ will pass through 3 multiple points. Consequently, $L_7$ must pass through $L_3\cap L_5$ so that $L_5$ will passes through 3 multiple points. Under such a intersection lattice structure, it is not difficult to see that the defining equation of the arrangement can be written as
\begin{equation} \label{n_4=2-4}\scriptstyle
xyz(x-z)(y-z)(x-t^2 z)(y-t z)(y-x)(y+\frac{1}{t}x-t z)(y+\frac{1}{t}(x-z))=0
 \end{equation}
 where $t=\frac{-1\pm\sqrt{5}}{2}$.~
From the equation, we know that the intersection lattice has real realizations (see Figure \ref{n4=2-real-excptional-2}. The line $L_{10}$ is at the infinity).
 \begin{figure}[htbp]
\centering
\subfigure{}{
\begin{tikzpicture}[domain=-2:4,scale=1]\tikzstyle{every node}=[font=\footnotesize]
\draw[domain=-1.5:2] plot (0, \x) node[above]{$L_4$};
\draw[domain=-1.5:2] plot (1, \x) node[above]{$L_5$};
\draw[domain=-1.5:2] plot ({(3-sqrt(5))/2}, \x) node[above]{$L_6$};

\draw[domain=2:-1] plot (\x, 0) node[left]{$L_3$};
\draw[domain=2:-1] plot (\x, 1) node[left]{$L_2$};
\draw[domain=2:-1] plot (\x, {(-1+sqrt(5))/2}) node[left]{$L_1$};

\draw[domain={-1:2}] plot (\x, \x) node[above]{$L_{9}$};
\draw[domain={1.2:-0.6}] plot (\x, {-((1+sqrt(5))/2)*\x+(-1+sqrt(5))/2}) node[left]{$L_8$};
\draw[domain={1.8:-0.2}] plot (\x, {-((1+sqrt(5))/2)*\x+(1+sqrt(5))/2}) node[left]{$L_7$};
\end{tikzpicture}
}
\hspace{2cm}
\subfigure{}{
\begin{tikzpicture}[domain=-2:4,scale=0.8]\tikzstyle{every node}=[font=\footnotesize]
\draw[domain=-2:1.8] plot (0, \x) node[above]{$L_4$};
\draw[domain=-2:1.8] plot (1, \x) node[above]{$L_5$};
\draw[domain=-2:1.8] plot ({(3+sqrt(5))/2}, \x) node[above]{$L_6$};

\draw[domain=3:-2] plot (\x, 0) node[left]{$L_3$};
\draw[domain=3:-2] plot (\x, 1) node[left]{$L_2$};
\draw[domain=3:-2] plot (\x, {(-1-sqrt(5))/2}) node[left]{$L_1$};

\draw[domain={-2:1.8}] plot (\x, \x) node[above]{$L_{9}$};
\draw[domain={-0.5:3}] plot (\x, {-((1-sqrt(5))/2)*\x+(-1-sqrt(5))/2}) node[right]{$L_8$};
\draw[domain={-2:3}] plot (\x, {-((1-sqrt(5))/2)*\x+(1-sqrt(5))/2}) node[right]{$L_7$};
\end{tikzpicture}
}
\caption{\label{n4=2-real-excptional-2}}
\end{figure}

Let $L_1\cap L_5$ be on  $L_8$.  Then $L_8$ must pass through $L_3\cap L_6$ so that $L_6$ will pass through 3 multiple points. Moreover, $L_8$ must pass through $L_2\cap L_4$ so that $L_4$ will pass through at least 3 multiple points. 
It is not difficult to see that the defining equation of the arrangement can be written as
\begin{equation*} \label{n_4=2-6}\scriptstyle
xyz(x-z)(y-z)(x-(t-1)z)(y-t z)(y-x)(y-(t-1)x - 2z)(y - (t -1)x- z)=0
 \end{equation*}
where $t=1\pm\sqrt{-1}$. So the complex conjugate quotiented moduli space $\mathcal{M}^c_{\mathcal{A}}$ is irreducible.

Let $L_1\cap L_6$ be on  $L_8$. Then $L_8$ must pass through $L_2\cap L_4$ so that $L_4$ will pass through at least 3 multiple points.  Consequently,  either $L_7$ or $L_8$ must pass through $L_3\cap L_5$ so that each of $L_5$ and $L_3$ will pass through 3 multiple points.
\begin{enumerate}
 \item Assume that $L_3\cap L_5$ is on  $L_8$. By writing down the defining equation, 
 we see that the moduli space is irreducible.
 \item Assume that $L_3\cap L_5$ is on  $L_7$. Then 
 defining equations can be written as follows:
 \begin{equation} \label{n_4=2-7}\scriptstyle
xyz(x-z)(y-z)(x-(1\pm\sqrt{2}/2)z)(y- (2\pm\sqrt{2})z)(y-x)(y\mp\sqrt{2}x\pm\sqrt{2}z)(y\mp\sqrt{2}x-z)=0.
 \end{equation}
 Thus, we have two real arrangements (see Figure \ref{n4=2-real-excptional-3} without $L_{10}$ which is the line at the infinity).
  \begin{figure}[htbp]
\centering
\subfigure{}{
\begin{tikzpicture}[domain=-2:4,scale=0.6]\tikzstyle{every node}=[font=\footnotesize]
\draw[domain=-1.5:4] plot (0, \x) node[above]{$L_4$};
\draw[domain=-1.5:4] plot (1, \x) node[above]{$L_5$};
\draw[domain=-1.5:4] plot ({1+sqrt(2)/2}, \x) node[above]{$L_6$};

\draw[domain=4:-1] plot (\x, 0) node[left]{$L_3$};
\draw[domain=4:-1] plot (\x, 1) node[left]{$L_2$};
\draw[domain=4:-1] plot (\x, {2+sqrt(2)}) node[left]{$L_1$};

\draw[domain={4:-1}] plot (\x, \x) node[left]{$L_{9}$};
\draw[domain={-0.1:3.85}] plot (\x, {sqrt(2)*\x-sqrt(2)}) node[above]{$L_7$};
\draw[domain={2:-1}] plot (\x, {sqrt(2)*\x+1}) node[left]{$L_8$};
\end{tikzpicture}
}
\hspace{2cm}
\subfigure{}{
\begin{tikzpicture}[domain=-2:4,scale=1]\tikzstyle{every node}=[font=\footnotesize]
\draw[domain=-1:2] plot (0, \x) node[above]{$L_4$};
\draw[domain=-1:2] plot (1, \x) node[above]{$L_5$};
\draw[domain=-1:2] plot ({1-sqrt(2)/2}, \x) node[above]{$L_6$};

\draw[domain=2:-1] plot (\x, 0) node[left]{$L_3$};
\draw[domain=2:-1] plot (\x, 1) node[left]{$L_2$};
\draw[domain=2:-1] plot (\x, {2-sqrt(2)}) node[left]{$L_1$};

\draw[domain={2:-1}] plot (\x, \x) node[left]{$L_{9}$};
\draw[domain={1.5:-0.4}] plot (\x, {-sqrt(2)*\x+sqrt(2)}) node[above]{$L_7$};
\draw[domain={1.3:-0.5}] plot (\x, {-sqrt(2)*\x+1}) node[left]{$L_8$};
\end{tikzpicture}
}
\caption{\label{n4=2-real-excptional-3}}
\end{figure}
 \end{enumerate}

\end{proof}

\section{Arrangements of 10 lines with a single quadruple point}

Let $\mathcal{A}$ be a non-reductive arrangement of 10 lines with a single quadruple point.  By Lemma~ \ref{lemma:intersection-formula} and Theorem \ref{thm:Hirzebruch}, we know that there are at most 12 triple points.

We say  that 2 multiple points of $\mathcal{A}$ are disjoint if they are not on the same line of $\mathcal{A}$. We say that 2 multiple points of $\mathcal{A}$ are adjoint if they are on the same line of $\mathcal{A}$.  Since $\mathcal{A}$ contains only 10 lines, it is not hard to see that there are at most 2 disjoint triple points apart from the quadruple point.

\subsection{Two disjoint triple points apart from the pencil of the quadruple point}

\begin{prop}Let $\mathcal{A}=\{L_1, L_2, \dots, L_{10}\}$ be a non-reductive line arrangement in $\mathbb{CP}^2$ with a unique quadruple point, say  $L_1\cap L_2\cap L_3\cap L_7$.
 Assume that there are $2$ disjoint triple points which are apart from $L_1\cup L_2\cup L_3\cup L_7$. Then either $\mathcal{M}^c_{\mathcal{A}}$ is irreducible, or $\mathcal{A}$ is isomorphic to the arrangements defined by equations \eqref{eq:n_4=1-1}.
\end{prop}
\begin{proof}Let $L_4\cap L_5\cap L_6$ and $L_8\cap L_9\cap L_{10}$ be the 2 triple points.
We claim that there are at least 2 of $L_8$, $L_9$ and $L_{10}$ that each passes through 3 intersection
 points of $(L_1\cup L_2\cup L_3\cup L_7) \cap (L_4\cup L_5\cup L_6)$.
In fact, there should be at least 8 triple points in $L_1\cup L_2\cup L_3\cup L_7$ so that each of the 4 lines will pass through at least 3 multiple points.  Since $L_8\cap L_9\cap L_{10}$ is a triple point, then the 8 points must be in $(L_1\cup L_2\cup L_3\cup L_7) \cap (L_4\cup L_5\cup L_6)$. Notice that each of $L_8$, $L_9$ and $L_{10}$ can pass through at most 3 of those intersection points.
If only one of $L_8$, $L_9$ and $L_{10}$ passes through 3 intersection points of $(L_1\cup L_2\cup L_3\cup L_7) \cap (L_4\cup L_5\cup L_6)$, then there will be at most 7 triple points in $L_1\cup L_2\cup L_3\cup L_7$.
Let $L_8$ and $L_9$ be the two lines  each of which passes through 3 intersection points of $(L_1\cup L_2\cup L_3\cup L_7) \cap (L_4\cup L_5\cup L_6)$.

We may assume that $L_4\cap L_7$, $L_5\cap L_3$ and $L_6\cap L_2$ are on $L_8$. By our assumption, the line $L_9$ should
 pass through one of $L_1\cap L_4$, $L_1\cap L_5$ and $L_1\cap L_6$ so that $L_1$ will pass through 3 multiple points.
 By switching the labels of $L_4$, $L_5$ and $L_6$, and others accordingly,
 we may assume that $L_1\cap L_4$ is on $L_9$.

By our assumption, we see that $L_5\cap L_9$ and $L_6\cap L_9$ should also be triple points. Then either  $L_2\cap L_5$ or $L_5 \cap L_7$ should be on  $L_9$. Correspondingly, $\{L_2\cap L_5, L_3\cap L_6\}\subset L_9$, $\{L_2\cap L_5, L_6 \cap L_7\} \subset
L_9$ or $\{L_5 \cap L_7, L_3\cap L_6\} \subset L_9$.  Notice that the last two cases are equivalent by a permutation $(5, 6)(2, 3)$.

Assume that $\{L_1\cap L_4, L_2\cap L_5, L_3\cap L_6\}$ is in $L_9$  (see Figure \ref{n4=1-1})  \begin{figure}[htbp]\centering
\begin{tikzpicture}[domain=-1:6, scale=0.6]\tikzstyle{every node}=[font=\footnotesize]
\draw[domain=-1:5] plot (0, \x) node[above]{$L_4$};
\draw[domain=-1:5] plot (1, \x) node[above]{$L_5$};
\draw[domain=-1:5] plot (2, \x) node[above]{$L_6$};

\draw[domain=4:-2] plot (\x, 3) node[left]{$L_1$};
\draw[domain=4:-2] plot (\x, 2) node[left]{$L_2$};
\draw[domain=4:-2] plot (\x, 1) node[left]{$L_3$};
\draw[domain=4:-2] plot (\x, 0) node[left]{$L_7$};

\draw[domain=-1:4] plot (\x, \x) node[above]{$L_8$};
\draw[domain=4:-1] plot (\x, -\x+3) node[above]{$L_9$};

\end{tikzpicture}
\caption{\label{n4=1-1}}
\end{figure}.
By automorphisms of the dual of the projective plane, we can write
$L_7=\{y=0\}$, $L_3=\{y=z\}$, $L_2=\{y=tz\}$, $L_4=\{x=0\}$, $L_5=\{x=z\}$, $L_6=\{x=tz\}$, $L_{8}=\{x=y\}$, $L_{9}=\{y=-x+(t+1)z)\}$, and $L_1=\{y=(t+1)z\}$, where $t$ is a complex number, and $t\neq 0, 1$.
Since $L_1$ passes through 3 multiple points by our assumption, then $L_{10}$ should pass through either $L_1\cap L_6$ or $L_1\cap L_5$. On the other hand, $L_8\cap L_9$ should be on $L_{10}$ by our assumption.  If $L_1\cap L_6$ is on $L_{10}$, then  $L_5 \cap L_7$ also should be on $L_{10}$ so that $L_7$ will pass through 3 multiple points.
If $L_1\cap L_5$ is on $L_{10}$, then  $L_6 \cap L_7$ should also be on $L_{10}$ so that $L_7$ will have 3 multiple points. Notice by the permutation $(1, 7)(2, 3)(8, 9)$, the 2 arrangements are lattice isomorphic. Let $L_1\cap L_5$ and $L_6\cap L_7$ be on $L_{10}$, then $L_{10}$ is defined by $y=\frac{t+1}{1-t}x+\frac{t^2+t}{t-1}z$.  We check that for any $t\neq 0, 1$, the intersection $L_8\cap L_9\cap L_{10}$ is always non-empty. Thus the arrangement is defined by the following equation,
\[\scriptstyle xy(x-z)(y-z)(x-tz)(y-tz)(y-(t+1)z)(x-y)(y+x-(t+1)z)((t-1)y+(t+1)(x-tz))=0,\]
where $t$ is a complex number and $t\neq 0, 1$. It is clear that the moduli space $\mathcal{M}_{\mathcal{A}}$ is irreducible. However, one can check that it is possible that $L_{10}$ will pass through 1 more triple point $L_3\cap L_4$. Under this additional condition, we get $t^2=-1$. Hence the quotient moduli space $\mathcal{M}^c_{\mathcal{A}}$ is irreducible.

Now assume that $\{L_1\cap L_4, L_2\cap L_5, L_6\cap L_7\}$ is in $L_9$  (see Figure \ref{n4=1-2}).  \begin{figure}[htbp]\centering
\begin{tikzpicture}[domain=-1:6, scale=0.6]\tikzstyle{every node}=[font=\footnotesize]
\draw[domain=-1:5] plot (0, \x) node[above]{$L_4$};
\draw[domain=-1:5] plot (1, \x) node[above]{$L_5$};
\draw[domain=-1:5] plot (2, \x) node[above]{$L_6$};

\draw[domain=5:-1] plot (\x, 4) node[left]{$L_1$};
\draw[domain=5:-1] plot (\x, 2) node[left]{$L_2$};
\draw[domain=5:-1] plot (\x, 1) node[left]{$L_3$};
\draw[domain=5:-1] plot (\x, 0) node[left]{$L_7$};

\draw[domain=-1:5] plot (\x, \x) node[above]{$L_8$};
\draw[domain=2.5:-0.5] plot (\x, -2*\x+4) node[left]{$L_9$};

\end{tikzpicture}
\caption{\label{n4=1-2}}
\end{figure}
By automorphisms of the dual plane of the projective plane, we can write
$L_7=\{y=0\}$, $L_3=\{y=z\}$, $L_2=\{y=tz\}$, $L_4=\{x=0\}$, $L_5=\{x=z\}$, $L_6=\{x=tz\}$, $L_{8}=\{x=y\}$, $L_{9}=\{y=\frac{t}{1-t}(x-tz)\}$, and $L_1=\{y=\frac{t^2}{t-1}z\}$, where $t$ is a complex number, and $t\neq 0, 1$.

Then either $L_1\cap L_5$ or $L_1\cap L_6$ is on
$L_{10}$ so that $L_1$ will pass through 3 multiple points. Similarly, either $L_3\cap L_4$ or $L_3\cap L_6$ is on $L_{10}$.
Thus $L_{10}$ should pass through one of the
following 3 sets of points $\{L_1\cap L_5, L_3\cap L_4\}$, $\{L_1\cap L_5, L_3\cap L_6\}$ or $\{L_1\cap L_6, L_3\cap L_4\}$. By the permutation
\[(1,7,2,3)(4, 9, 6, 8,5, 10)\]
 the case that  $\{L_1\cap L_6, L_3\cap L_4\}\subset L_{10}$ can be identified with the case that $\{L_1\cap L_5, L_3\cap L_6\}\subset L_{10}$.
 Now we compute the defining equation of $L_{10}$.

  If $\{L_1\cap L_5, L_3\cap L_4\}$ is in $L_{10}$,
then $L_{10}=\{y=\frac{t^2-t+1}{t-1}x+z\}$. Since $L_8\cap L_9\cap L_{10}$ is non-empty, then $t$ must be a root of the following polynomial $t^4-2t^3+4t^2-3t+1$.
Hence, $\mathcal{A}$ is defined by the following equation
\begin{equation}\label{eq:n_4=1-1}
\scriptstyle xy(x-z)(y-z)(x-tz)(y-tz)((t-1)y-t^2z)(x-y)(tx+(t-1)y-t^2z)((t^2-t+1)x-(t-1)y+(t-1)z)=0\end{equation}
where $t$ is a solution of the following equation $t^4-2t^3+4t^2-3t+1=0$.

If $\{L_1\cap L_5, L_3\cap L_6\}$ is in $L_{10}$,
then $L_{10}=\{y=-\frac{t^2-t+1}{(t-1)^2}(x-tz)+z\}$. Since $L_8\cap L_9\cap L_{10}$ is non-empty, then $t$ must be a root of the following polynomial $2t^3-4t^2+3t-1=(t-1)(2t^2-2t+1)$. However, by our assumption $t\neq 1$.
Hence, $\mathcal{A}$ is defined by the following equation
\begin{equation*}\label{eq:n_4=1-2} \scriptstyle xy(x-z)(y-z)(x-tz)(y-tz)((t-1)y-t^2z)(x-y)(tx+(t-1)y-t^2z)((t^2-t+1)x+(t-1)y-(t^3-t+1)z)=0\end{equation*}
where $t$ is a solution of the following equation $2t^2-2t+1=0$. Therefore, the quotient moduli space $\mathcal{M}^c(\mathcal{A})$ is irreducible.
\end{proof}

\subsection{All triple points apart from the pencil of the quadruple point have an adjoint point}

\begin{thm}Let $\mathcal{A}=\{L_1, L_2,\dots, L_{10}\}$ be a non-reductive line arrangement in $\mathbb{CP}^2$ with a quadruple point $L_1\cap L_2\cap L_3\cap L_4$. Assume that $L_5\cap L_6\cap L_7$ is a triple point apart from $L_1\cup L_2\cup L_3\cup L_4$ and all triple points apart from $L_1\cup L_2\cup L_3\cup L_4$ are on $L_5\cup L_6\cup L_7$.  Then one of the following holds:
\begin{enumerate}
\item the moduli space $\mathcal{M}_{\mathcal{A}}$ is irreducible;
\item the quotient moduli space $\mathcal{M}^c_{\mathcal{A}}$ is irreducible;
\item the arrangement $\mathcal{A}$ is defined by one of the equations \eqref{eq:n_4=1-3}, \eqref{5.2-9-triples-1}, \eqref{eq:5.2-9-triples-2} and \eqref{eq:5.2-9-triples-3}.
\end{enumerate}

\end{thm}
 \begin{proof}
 It is clear that there should be at least 9 triple points so that each line passes through 3 multiple points. Note that all triple points will be in $(L_8\cup L_9\cup L_{10})\cap (L_1\cup L_2\cup \cdots\cup L_7)$, except $L_5\cap L_6\cap L_7$.  By B\'ezout's Theorem, the intersection number of $(L_1\cup L_2\cup \cdots\cup L_7)$ and $(L_8\cup L_9\cup L_{10})$, is 21. Since the intersection multiplicity of a triple point is $2$, then there will be at most 10 triple points in $L_8\cup L_9\cup L_{10}\cap (L_1\cup L_2\cup \cdots\cup L_7)$. Plus the triple point $L_5\cap L_6\cap L_7$, we will have at most 11 triple points.

\paragraph{Case 1:} Assume that there are 11 triple points. We claim that except for the triple point $L_5\cap L_6\cap L_7$, all triple points are in $L_1\cup L_{2}\cup L_{3} \cup L_4$. Assume in contrary that $L_8\cap L_9$ is a triple point in $L_7$. Then each of $L_8$ and $L_9$ will pass through at most 2 more triple points apart from $L_{10}$. On $L_{10}$, we can have at most 4 triple points. Totally, there will be  at most 10 triple points.  Since all the triple points except
$L_5\cap L_6\cap L_7$ are in $L_1\cup L_2 \cup L_{3} \cup L_4$, then there are two of $L_1$, $L_2$, $L_3$ and $L_4$ such that each passes through 3 triple points which must be in $(L_1\cup L_{2}\cup L_{3} \cup L_4)\cap (L_5\cup L_6\cup L_7)$. Therefore, by lattice isomorphism, we may assume that the arrangement contains the sub-arrangement in Figure \ref{n_4=1-n_3=11}.
\begin{figure}[htbp]\centering
\begin{tikzpicture}[domain=-2:4, scale=0.6]\tikzstyle{every node}=[font=\footnotesize]
\draw[domain=-3:2] plot (0, \x) node[above]{$L_5$};
\draw[domain=-3:2] plot (1, \x) node[above]{$L_6$};
\draw[domain=-3:2] plot (2, \x) node[above]{$L_7$};

\draw[domain=3:-1] plot (\x, 0) node[left]{$L_3$};
\draw[domain=3:-1] plot (\x, -1) node[left]{$L_4$};

\draw[domain=2.5:-1] plot (\x, {\x-1}) node[left]{$L_8$};
\draw[domain=2.5:-1] plot (\x, {\x-2}) node[left]{$L_9$};
\draw[domain=3:-1] plot (\x, {-0.5*\x}) node[above]{$L_{10}$};
\end{tikzpicture}
\caption{\label{n_4=1-n_3=11}}
\end{figure}
Each of $L_1$ and $L_2$ should pass through 2 triple points. Since there are at most 3 double points in $L_8\cup L_9\cup L_{10}$, then at least one of $L_8$, $L_9$ and $L_{10}$ should pass through a point in $(L_1\cup L_2)\cap (L_5\cup L_6\cup L_7)$. Assume that $L_8$ passes through $L_2\cap L_7$. If $L_9$ or $L_{10}$ passes through a double point of $L_2\cap (L_5\cup L_6\cup L_7)$, then the arrangement contains a MacLane arrangement as a sub-arrangement. It is not hard to verify that the quotient moduli space $\mathcal{M}^c_{\mathcal{A}}$ is irreducible.  Assume that the arrangement does not contain a MacLane arrangement. Then $L_2\cap L_9\cap L_{10}$ must be a triple point so that $L_2$ will pass through 2 triple points. Then the arrangement contains a sub-arrangement $\mathcal{A}'=\mathcal{A}\setminus\{L_1\}$. which is lattice isomorphic to $\mathcal{A}^{\pm\sqrt{-1}}$. Therefore, the quotient moduli space $\mathcal{M}^c_{\mathcal{A}}$ is irreducible.

\paragraph{Case 2:} Assume that there are 10 triple points.
There are two main subcases:
\paragraph{Subcase 1:}
Assume that there is another triple point, say $L_7\cap L_8\cap L_9$ apart from $L_1\cup L_2\cup L_3\cup L_4$. Notice that each of $L_8$ and $L_9$ passes through at most 2 more triple points apart from $L_{10}$ which are on $L_5\cup L_6$. On $L_{10}$ there can be at most 4 triple points which are on $L_1\cup L_2\cup L_3\cup L_4$ (see for instance Figure \ref{n_4=1-3}).
Then each of $L_8$ and $L_9$ should pass through 2 points in $(L_1\cup L_{2} \cup L_{3}\cup L_4)\cap (L_5\cup L_6)$ and $L_{10}$ should pass through 4 triple points so that there will be 10 triple points.  Since each of $L_1$, $L_2$, $L_3$ and $L_4$ should pass through 2 triple points (so that each will pass through 3 multiple points), then each of the four lines $L_1$, $L_2$, $L_3$ and $L_4$ should pass through exactly 1 of the 4 triple points in $(L_1\cup L_2\cup L_3\cup L_4)\cap (L_5\cup L_6)\cap (L_8\cup L_9)$. By lattice isomorphism, we may assume that the arrangement contains the sub-arrangement in Figure \ref{n4=1-3}.
\begin{figure}[htbp]\centering\begin{tikzpicture}[domain=-1:6, scale=0.6]\tikzstyle{every node}=[font=\footnotesize]
\draw[domain=-1:5] plot (0, \x) node[above]{$L_5$};
\draw[domain=-1:5] plot (1, \x) node[above]{$L_6$};
\draw[domain=-1:5] plot (3, \x) node[above]{$L_7$};

\draw[domain=4:-2] plot (\x, 1.5) node[left]{$L_3$};
\draw[domain=4:-2] plot (\x, 2.4) node[left]{$L_2$};
\draw[domain=4:-2] plot (\x, 3.1) node[left]{$L_1$};
\draw[domain=4:-2] plot (\x, 0) node[left]{$L_4$};

\draw[domain=3.57:-2] plot (\x, {0.7*\x+2.4}) node[left]{$L_8$};
\draw[domain=3.33:-0.66] plot (\x, {1.5*\x}) node[left]{$L_9$};

\end{tikzpicture}
\caption{\label{n4=1-3}}
\end{figure}
We may assume that $L_1\cap L_7\cap L_{10}$ is a triple point so that $L_7$ will pass through 3 multiple points.
The defining equations of $L_1$, $L_2$, $\dots$, $L_9$ can be written as
$L_1=\{y=z\}$, $L_2=\{y=0\}$, $L_3=\{y=t_1z\}$, $L_4=\{y=t_2z\}$, $L_5=\{x=0\}$, $L_6=\{x=z\}$, $L_7=\{x=t_3z\}$, $L_8=\{y=x\}$ and $L_9=\{y=(t_1-t_2)x+t_2z\}$, where $t_1, t_2, t_3\in\mathbb{C}\setminus\{0, 1\}$, $t_1\neq t_2$, and satisfy the equation \[t_3(t_1-t_2-1)+t_2=0.\]

Recall that $L_{10}$ should pass through 4 triple points so that there will be 10 triple points. Since each line should pass through 3 multiple points, the possibilities are as follows.

\begin{enumerate}
\item $\{L_2\cap L_6, L_3\cap L_5, L_4\cap L_8\}\subset L_{10}$ (see Figure \ref{n_4=1-3}).

\begin{figure}[htbp]
\begin{minipage}{0.5\textwidth}
\centering\begin{tikzpicture}[domain=-1:6,scale=0.6]\tikzstyle{every node}=[font=\footnotesize]
\draw[domain=-4:3] plot (0, \x) node[above]{$L_5$};
\draw[domain=-4:3] plot (1, \x) node[above]{$L_6$};
\draw[domain=-4:3] plot ({(-0.75487766624669284
-1)/-0.75487766624669284
}, \x) node[above]{$L_7$};
\draw[domain=3:-4] plot (\x, -0.75487766624669284
) node[left]{$L_3$};
\draw[domain=3:-4] plot (\x, 1) node[left]{$L_1$};
\draw[domain=3:-4] plot (\x, 0) node[left]{$L_2$};
\draw[domain=3:-4] plot (\x, {-0.75487766624669284
/(-0.75487766624669284
+1)}) node[left]{$L_4$};
\draw[domain=-3.8:3] plot (\x, \x) node[right]{$L_8$};
\draw[domain=2.6:-0.35] plot (\x, {(-0.75487766624669284
+0.75487766624669284
/(-0.75487766624669284
+1))*\x-0.75487766624669284
/(-0.75487766624669284
+1)}) node[left]{$L_9$};
\draw[domain=-4:3] plot (\x, {0.75487766624669284
*\x-0.75487766624669284}) node[right]{$L_{10}$};
\end{tikzpicture}
\caption{\label{n_4=1-3}}
\end{minipage}
\begin{minipage}{0.5\textwidth}
\centering
\begin{tikzpicture}[domain=-1:6,scale=0.5]\tikzstyle{every node}=[font=\footnotesize]
\draw[domain=-4:4] plot (0, \x) node[above]{$L_5$};
\draw[domain=-4:4] plot (1, \x) node[above]{$L_6$};
\draw[domain=-4:4] plot (3, \x) node[above]{$L_7$};

\draw[domain=4:-4] plot (\x, -1) node[left]{$L_3$};
\draw[domain=4:-4] plot (\x, 1) node[left]{$L_1$};
\draw[domain=4:-4] plot (\x, 0) node[left]{$L_2$};
\draw[domain=4:-4] plot (\x, -3) node[left]{$L_4$};

\draw[domain=-3.8:4] plot (\x, \x) node[right]{$L_8$};
\draw[domain=3.5:-0.4] plot (\x, {2*\x-3}) node[left]{$L_9$};
\draw[domain=-4:4] plot (\x, {(2/3)*(\x-3/2)}) node[right]{$L_{10}$};
\end{tikzpicture}
\caption{\label{n_4=1-6}}
\end{minipage}
\end{figure}

Then $L_{10}$ is defined by $y=-t_1x+t_1z$ and $t_1, t_2, t_3$ satisfy the following two equations:
\[1=t_1-t_1t_3
, \textup{~~and~~} t_2=t_1-t_1t_2.
\]
Simplify the system of those three equations, we obtain that $t_3=\frac{t_1-1}{t_1}$, $t_2=\frac{t_1}{t_1+1}$ and $t_1^3-t_1^2+1=0$. A defining equation can be written as:
\begin{equation}\label{eq:n_4=1-3}
\scriptstyle
xy(x-z)(y-z)(x-y)(x-\frac{t-1}{t}z)(y-t z)(y-\frac{t}{t+1}z)(y-\frac{t^2+t+1}{t+1}x-\frac{t}{t+1}z)(y-\frac{t^2-t+1}{t^2}(x+\frac{1}{t}z))=0
\end{equation}
where $t$ satisfies $t^3-t^2+1=0$. 

\item $\{L_2\cap L_9, L_3\cap L_5, L_4\cap L_8\}\subset L_{10}$ (see Figure \ref{n_4=1-6}).
One can verify that in this case $t_1=-1$, $t_2=-3$,  and $t_3=3$. Therefore, the moduli space $\mathcal{M}_{\mathcal{A}}$ is irreducible.

\item $\{L_2\cap L_9, L_3\cap L_8, L_4\cap L_6\}\subset L_{10}$.  By a permutation $(5,9)(6,8)(2, 3)$, we see that this arrangement is a lattice isomorphic to the one shown in Figure \ref{n_4=1-3}.

\item $\{L_2\cap L_9, L_3\cap L_5, L_4\cap L_6\}\subset L_{10}$. By a permutation $(5,9)(6,8)(2, 3)$, we see that this arrangement is lattice isomorphic to the one shown in Figure \ref{n_4=1-6}.

\end{enumerate}

\paragraph{Subcase 2:}
Assume that all triple points, except $L_5\cap L_6\cap L_7$, are on $L_1\cup L_2\cup L_3\cup L_4$. Since there are 9 triple points on $L_1\cup L_2\cup L_3\cup L_4$, then one of $L_1$, $L_2$, $L_3$ and $L_4$, say $L_1$, must pass through 3 triple points and each of the other 3 lines $L_2$, $L_3$ and $L_4$ should pass through exactly 2 triple points. Note that there are at most 3 triple points, $L_8\cap L_9$, $L_8\cap L_{10}$ and $L_9\cap L_{10}$, apart from $L_5\cup L_6\cup L_7$.  We will classify arrangements in this case according to the number of triple points apart from $L_5\cup L_6\cup L_7$.

\begin{enumerate}
\item There are 3 triple points apart from $L_5\cup L_6\cup L_7$.
 In this case, $L_8\cap L_9$, $L_8\cap L_{10}$ and $L_9\cap L_{10}$ are triple points in $(L_2\cup L_3\cup L_4)\setminus (L_5\cup L_6\cup L_7)$.  It is not hard to see that each of $L_8$, $L_9$ and $L_{10}$ passes through exactly 4 triple points, and each of $L_5$, $L_6$ and $L_7$ passes through exactly 3 triple points.  We may assume that $L_1\cap L_5\cap L_8$, $L_4\cap L_7\cap L_8$, $L_2\cap L_8\cap L_9$ and $L_3\cap L_{10}$ are the triple points on $L_8$. Moreover, we may assume that $L_1\cap L_6\cap L_9$, $L_1\cap L_7\cap L_{10}$ and $L_4\cap L_9\cap L_{10}$ are triple points.  Since each line should pass through at least 3 multiple points, then $L_3\cap L_5\cap L_9$ and $L_2\cap L_6\cap L_{10}$ should be triple points (see Figure \ref{n4=1-4}).
  \begin{figure}[htbp]
 \begin{minipage}{0.5\textwidth}
\centering\begin{tikzpicture}[domain=-1:6, scale=0.6]\tikzstyle{every node}=[font=\footnotesize]
\draw[domain=-1:6] plot (0, \x) node[above]{$L_5$};
\draw[domain=-1:6] plot (1, \x) node[above]{$L_6$};
\draw[domain=-1:6] plot (1.8019377358048223, \x) node[above]{$L_7$};

\draw[domain=3:-3] plot (\x, {(1.8019377358048223+1)/1.8019377358048223}) node[left]{$L_3$};
\draw[domain=3:-3] plot (\x, {1.8019377358048223^2+1.8019377358048223}) node[left]{$L_4$};
\draw[domain=3:-3] plot (\x, 0) node[left]{$L_1$};
\draw[domain=3:-3] plot (\x, 1) node[left]{$L_2$};

\draw[domain=2.1:-0.3] plot (\x, {(1.8019377358048223+1)*\x}) node[left]{$L_8$};
\draw[domain=1.5:-2.8] plot (\x, {-(1.8019377358048223+1)/1.8019377358048223*\x+(1.8019377358048223+1)/1.8019377358048223}) node[above]{$L_9$};
\draw[domain=2.3:-3] plot (\x, {1/(1-1.8019377358048223)*\x-1.8019377358048223/(1-1.8019377358048223)}) node[left]{$L_{10}$};
\end{tikzpicture}
\caption{\label{n4=1-4}}
\end{minipage}
\begin{minipage}{0.5\textwidth}
\centering\begin{tikzpicture}[domain=-1:6, scale=0.8]\tikzstyle{every node}=[font=\footnotesize]
\draw[domain=-1:4] plot (0, \x) node[above]{$L_5$};
\draw[domain=-1:4] plot (1, \x) node[above]{$L_6$};
\draw[domain=-1:4] plot (2, \x) node[above]{$L_7$};

\draw[domain=3:-1] plot (\x, 2) node[left]{$L_3$};
\draw[domain=3:-1] plot (\x, 3) node[left]{$L_4$};
\draw[domain=3:-1] plot (\x, 0) node[left]{$L_1$};
\draw[domain=3:-1] plot (\x, 1) node[left]{$L_2$};

\draw[domain=1.2:-0.5] plot (\x, \x) node[left]{$L_8$};
\draw[domain=0.8:2.3] plot (\x, -2*\x+4) node[right]{$L_9$};
\draw[domain=1.3:0.7] plot (\x, \x-1) node[below]{$L_{10}$};
\end{tikzpicture}
\caption{\label{n_4=1-8}}
\end{minipage}
\end{figure}
 By automorphism of the dual of the projective plane, we can write down the equations of the lines as follows:
$L_1=\{y=0\}$, $L_2=\{y=z\}$, $L_3=\{y=t_2z\}$, $L_4=\{y=t_3z\}$, $L_5=\{x=0\}$, $L_6=\{x=z\}$, $L_7=\{x=t_1z\}$, $L_8=\{y=\frac{t_3}{t_1}x\}$,  $L_9=\{y=-t_2x+t_2z\}$, and $L_{10}=\{y=\frac{1}{1-t_1}(x-t_1z)\}$ where $t_1$, $t_2$, $t_3$ are complex numbers and satisfy the following equations: $t_2=\frac{t_1+1}{t_1}$, $t_3=t_1^2+t_1$, and $(t_1+1)(t_1^3-t_1^2-2t_1+1)=0$. Notice that $t_1+1\neq 0$, otherwise $t_2=t_3=0$. So a defining equation can be written as
\begin{equation}\label{5.2-9-triples-1} \scriptstyle
xy(x-z)(y-z)(x-t z)(y-\frac{t+1}{t}z)(y-(t^2+t)z)(y-(t+1)x)(y+\frac{t+1}{t}x-\frac{t+1}{t}z)(y-\frac{1}{1-t}(x-t z))=0,
\end{equation}
where $t$ satisfies $t^3-t^2-2t+1$=0.

\item There are 2 triple points apart from $L_5\cup L_6\cup L_7$. Then there are 8 triple points on $L_5$, $L_6$ and $L_7$. So one of them should pass through 4 triple points including  $L_5\cap L_6\cap L_7$.  Let $L_6$ be the line passing though 4 triple points. On the other hand, one of $L_1$, $L_2$, $L_3$ and $L_4$ should pass through 3 triple points, since there are 9 triple points on $L_1\cup L_2\cup L_3\cup L_4$. By lattice isomorphism, we may assume that $L_1\cap L_5\cap L_8$, $L_1\cap L_6\cap L_{10}$, $L_1\cap L_9\cap L_7$, $L_2\cap L_6\cap L_8$ and $L_3\cap L_6\cap L_9$ are triple points (see Figure \ref{n_4=1-8}).

 By automorphism of the dual of the projective plane, we can write down the equations of lines $L_1$, $L_2$, $\dots$, $L_9$ as follows:
$L_1=\{y=0\}$, $L_2=\{y=z\}$, $L_3=\{y=t_2z\}$, $L_4=\{y=t_3z\}$, $L_5=\{x=0\}$, $L_6=\{x=z\}$, $L_7=\{x=t_1z\}$, $L_8=\{y=x\}$,  $L_9=\{y=\frac{t_2}{1-t_1}(x-t_1z)\}$ where $t_1, t_2, t_3\in\mathbb{C}\setminus\{0, 1\}$, $t_1\neq t_2$.

\begin{itemize}
\item $L_8\cap L_9$ is a triple point. Then it must be on $L_4$.  It follows that
\[t_1t_2-t_1t_3-t_2t_3+t_3=0.\]
By our assumption, either $L_8\cap L_{10}$ or $L_9\cap L_{10}$ should be a triple point. Up to a permutation $(8,9)(2,3)(5,7)$, we may assume that $L_9\cap L_{10}$ is the triple point. Then $L_{10}$ can be written as $L_{10}=\{y=\frac{t_2}{(1-t_1)(1-t_2)}(x-z)\}$. Note that $L_{10}$ should also pass through $L_4\cap L_5$ or $L_3\cap L_5$ so that $L_5$ will pass through 3 multiple points.

\begin{enumerate}
\item If $L_4\cap L_5$ is on $L_{10}$, then either $L_3\cap L_{10}$ or $L_3\cap L_8$ is on $L_7$ so that $L_7$ will pass through 3 multiple points.

 If $L_3\cap L_7\cap L_{10}$ is a triple point, then $t_1$, $t_2$ and $t_3$ satisfy 2 more equations:
$t_3=-\frac{t_2}{(1-t_1)(1-t_2)}$ and $t_2=\frac{t_2}{(1-t_1)(1-t_2)}(t_1-1)$. Simplifying the 3 equations, we have $t_2=2$, $t_1=\pm\sqrt{-1}$ and $t_3=1\pm\sqrt{-1}$ which implies that $\mathcal{M}^c_{\mathcal{A}}$ is irreducible.

If $L_3\cap L_7\cap L_8$ is a triple points, then $t_1$, $t_2$ and $t_3$ must also satisfy the following 2 equations:
$t_3=-\frac{t_2}{(t_1-1)(t_2-1)}$ and $t_1=t_2$. Simplifying the 3 equations, we have $t_1=t_2$, $t_3=-\frac{t_1}{(t_1-1)^2}$ and $t_1^3-2t_1^2+3t_1-1=0$. The defining equation can be written as
\begin{equation}\label{eq:5.2-9-triples-2}
\scriptstyle
xy(x-z)(y-z)(x-tz)(y-tz)(y+\frac{t}{(t-1)^2}z)(y-x)(y-\frac{t}{1-t}(x-tz))(y-\frac{t}{(1-t)^2}(x-z)),
\end{equation}
where $t$ satisfies $t^3-2t^2+3t-1=0$.
\item
 If $L_3\cap L_5$ is on $L_{10}$, then $L_4\cap L_7$ must be on $L_{10}$ so that there will be 10 triple points and each of $L_4$ and $L_7$ passes through at least 3 multiple points. Then we have two more equations:
$t_2=-\frac{t_2}{(1-t_1)(1-t_2)}$ and $t_3=\frac{t_2}{(1-t_1)(1-t_2)}(t_1-1)$. Simplifying those equations, we have $t_1=\frac{t_2-2}{t_2-1}$, $t_3=\frac{t_2}{t_2-1}$, and $t_2=\frac{3}{2}$. Therefore, the moduli space $\mathcal{M}_{\mathcal{A}}$ is irreducible.
\end{enumerate}

\item $L_8\cap L_9$ is not a triple point. Then $L_8\cap L_{10}$ and $L_9\cap L_{10}$ must be triple points. The equation of $L_{10}$ can be written as $L_{10}=\{y=\frac{t_2}{(1-t_1)(1-t_2)}(x-z)\}$. Since $L_{10}$ passes through 3 triple points, then $t_1=2$.  Checking out triple points on $L_4$, we see that either $L_8\cap L_4$ or $L_9\cap L_4$ should be a triple point. We may assume that $L_9\cap L_4\cap L_5$ is a triple point. Consequently, we must have  $t_3=\frac{t_1t_2}{t_1-1}$. Now we notice that there must be another triple point on $L_4$.  Either $L_4\cap L_7\cap L_8$ or $L_4\cap L_7\cap L_{10}$ is a triple point.
\begin{enumerate}
\item Assume that $L_4\cap L_7\cap L_8$ is a triple. Then $t_3=t_1$. It turns out that $t_2=1$. This implies that $L_2$ and $L_3$ coincide, which can not hold.
\item Assume that $L_4\cap L_7\cap L_{10}$ is a triple. Then $t_3=\frac{t_2}{t_2-1}$. It follows that $t_1=2$, $t_2=\frac{3}{2}$ and $t_3=3$.
\end{enumerate}

\end{itemize}

\item  Assume that only 1 triple point is apart from $L_5\cup L_6\cup L_7$. Let $L_8\cap L_9$ be the triple point apart from $L_5\cup L_6\cap L_7$. Since there are 9 triple points on $(L_8\cup L_9\cup L_{10})\cap (L_1\cup L_2\cup L_3
\cup L_4)$ by our assumption, we may assume that each of $L_8$ and $L_{10}$ passes through 3 points of $(L_1\cup L_2\cup L_3\cup L_4)\cap (L_5\cup L_6\cup L_7)$. Moreover, we may assume that $L_1\cap L_8$, $L_1\cap L_9$ and $L_1\cap L_{10}$ are triple points.
By lattice isomorphism, we may assume that $L_1\cap L_5\cap L_8$, $L_2\cap L_6\cap L_8$, $L_3\cap L_7\cap L_8$, $L_1\cap L_9$, $L_1\cap L_{10}$, $L_2\cap L_9$ and $L_4\cap L_8\cap L_9$  are triple points.  Now we determine the arrangements according to possible incidence on $L_9$ and $L_{10}$.
\begin{enumerate}
\item $L_2\cap L_9\in L_5$ and $L_1\cap L_9\in L_6$.  Defining equations of the lines $L_1$,$\dots$, $L_9$ can be written as follows:
$L_1=\{y=0\}$, $L_2=\{y=z\}$, $L_3=\{y=t_1z\}$, $L_4=\{y=\frac{1}{2}z\}$, $L_5=\{x=0\}$, $L_6=\{x=z\}$, $L_7=\{x=t_1z\}$, $L_8=\{y=x\}$ and $L_9=\{x+y=z\}$, where $t_1 \in \mathbb{C}\setminus\{0 ,1/2, 1\}$.
It is not hard to see that $L_{10}$  should contain either
$\{L_1\cap L_7, L_3\cap L_5, L_4\cap L_6\}$ or $\{L_1\cap L_7, L_3\cap L_6, L_4\cap L_5\}$.
\begin{enumerate}
\item $\{L_1\cap L_7, L_3\cap L_5, L_4\cap L_6\}\subset L_{10}$. Then $L_{10}=\{x+y=t_1z\}$ and $L_4\cap L_6=(1, \frac{1}{2}, 1)\in L_{10}$. It turns out that  $t_1=\frac{3}{2}$.

\item $\{L_1\cap L_7, L_3\cap L_6, L_4\cap L_5\}\subset L_{10}$. Then $L_{10}=\{y=-\frac{1}{2t_1}x+\frac{z}{2}\}$ and $L_3\cap L_6=(1, t_1, 1)\in L_{10}$. It turns out that $2t_1^2-t_1+1=0$. Therefore $\mathcal{M}^c_{\mathcal{A}}$ is irreducible.
\end{enumerate}

\item $L_2\cap L_9\in L_5$ and $L_1\cap L_9\in L_7$.
Defining equations of the lines $L_1$,$\dots$, $L_9$ can be written as follows:
$L_1=\{y=0\}$, $L_2=\{y=z\}$, $L_3=\{y=t_1z\}$, $L_4=\{y=t_2z\}$, $L_5=\{x=0\}$, $L_6=\{x=z\}$, $L_7=\{x=t_1z\}$, $L_8=\{y=x\}$ and $L_9=\{y=-\frac{1}{t_1}x+z\}$, where $t_1\neq t_2\in \mathbb{C}\setminus\{0, 1\}$ satisfy a equation $t_2=-\frac{t_2}{t_1}+1$.
It is not hard to see that $L_{10}$  should pass through $L_1\cap L_6$, $L_3\cap L_5$ and $L_4\cap L_7$.
Then a defining equation of $L_{10}$ can be written as $y=-t_1x+t_1z$. Moreover, since $L_4\cap L_7=(t_1, t_2, 1)\in L_{10}$, then $t_2=-t_1^2+t_1$. Together with the previous equation, we get $t_1=0$. Then there is no such arrangement.

\item $L_2\cap L_9\in L_7$, then $L_1\cap L_9\in L_6$.
Defining equations of the lines $L_1$,$\dots$, $L_9$ can be written as follows:
$L_1=\{y=0\}$, $L_2=\{y=z\}$, $L_3=\{y=t_1z\}$, $L_4=\{y=t_2z\}$, $L_5=\{x=0\}$, $L_6=\{x=z\}$, $L_7=\{x=t_1z\}$, $L_8=\{y=x\}$ and $L_9=\{y=\frac{1}{t_1-1}(x-z)\}$, where $t_1\neq t_2\in \mathbb{C}\setminus\{0, 1\}$ satisfy a equation $t_2=\frac{t_2-1}{t_1-1}$.
Now consider $L_{10}$. It is not hard to see that $L_{10}$  should contains either
$\{L_1\cap L_7, L_3\cap L_6, L_4\cap L_5\}$ or $\{L_1\cap L_7, L_3\cap L_5, L_4\cap L_6\}$.
\begin{enumerate}
\item $\{L_1\cap L_7, L_3\cap L_6, L_4\cap L_5\}\subset L_{10}$. Then $L_{10}=\{y=-\frac{t_2}{t_1}x+t_2z\}$ and $L_4\cap L_6=(1, t_1, 1)\in L_{10}$. Then $t_2=t_1^2+1$ and $t_1^3-2t_1^2+t_1-1=0$. The defining equation of $\mathcal{A}$ can be written as
\begin{equation}\label{eq:5.2-9-triples-3}\scriptstyle
xy(x-y)(x-z)(y-z)(x-tz)(y-tz)(y-\frac{t^2}{t-1}z)(y-\frac{1}{t-1}(x-z))(y+\frac{t^2}{t(t-1)}x-\frac{t^2}{t-1}z),
\end{equation}
where $t$ satisfies $t^3-2t^2+t-1=0$.
 The real arrangement is shown in Figure \ref{n4=1-5}.
 \begin{figure}[htbp]\centering\begin{tikzpicture}[domain=-1:6, scale=0.6]\tikzstyle{every node}=[font=\footnotesize]
\draw[domain=-2:5] plot (0, \x) node[above]{$L_5$};
\draw[domain=-2:5] plot (1, \x) node[above]{$L_6$};
\draw[domain=-2:5] plot (1.7548776662466927
, \x) node[above]{$L_7$};

\draw[domain=5:-1] plot (\x, 1.7548776662466927
) node[left]{$L_3$};
\draw[domain=5:-1] plot (\x, {1.7548776662466927
^2/(1.7548776662466927
-1)}) node[left]{$L_4$};
\draw[domain=5:-1] plot (\x, 0) node[left]{$L_1$};
\draw[domain=5:-1] plot (\x, 1) node[left]{$L_2$};

\draw[domain=5:-1] plot (\x, \x) node[left]{$L_8$};
\draw[domain=4.7:-.4] plot (\x, {1/(1.7548776662466927
-1)*(\x-1)}) node[left]{$L_9$};
\draw[domain=2.6:-0.4] plot (\x, {-1.7548776662466927
/(1.7548776662466927
-1)*\x+1.7548776662466927
^2/(1.7548776662466927
-1)}) node[left]{$L_{10}$};
\end{tikzpicture}
\caption{\label{n4=1-5}}
\end{figure}
\item $\{L_1\cap L_7, L_3\cap L_5, L_4\cap L_6\}\subset L_{10}$. Then $L_{10}=\{x+y=t_1z\}$ and $L_4\cap L_6=(1, t_2, 1)\in L_{10}$. Then $t_1$ and $t_2$ satisfy the following equations: $t_1^2-3t_1+3=0$ and $t_2=t_1-1$. Therefore, the quotient moduli space $\mathcal{M}^c_{\mathcal{A}}$ is irreducible.
\end{enumerate}
\end{enumerate}

\item  All triple points are on $L_5\cup L_6\cup L_7$.  Then each of the lines $L_8$, $L_9$ and $L_{10}$ passes exactly 3 triple points and the triple points are in $(L_1\cup L_2\cup L_3\cup L_4)\cap (L_5\cup L_6\cup L_7)$. We may assume that $L_8$ passes through $L_1\cap L_5$, $L_2\cap L_6$ and $L_3\cap L_7$.  Let $L_9$ be the line passing through $L_1\cap L_6$ and $L_{10}$ be the line passing through $L_1\cap L_7$. Notice that each of $L_5$, $L_6$ and $L_7$ passes through 4 triple points so that all 10 triple points will be on them.  Then $L_{10}$ must pass through either $L_3\cap L_6$ or $L_4\cap L_6$.
\begin{enumerate}
\item
If $L_3\cap L_6\in L_{10}$, then $L_{10}$ must pass through $L_4\cap L_5$.
Consequently, $L_9$ should pass through $L_2\cap L_5$ and $L_4
\cap L_7$.
Then the defining equation can be written as
 \begin{equation*}\label{eq:str}
 \scriptstyle
 xy(x-z)(y-z)(x-y)(x-t_1z)(y-t_1z)(y-t_2 z)(y+x-z)(y-\frac{t_1}{1-t_1}(x-t_1z))  \tag{$\star$}
 \end{equation*}
 where $t_1=\frac{1\pm\sqrt{-1}}{2}$ and $t_2=1-t_1=\frac{1\mp\sqrt{-1}}{2}$.

\item
 If $L_4\cap L_6\in L_{10}$, then $L_{10}$ must pass through either $L_3\cap L_5$ or $L_2\cap L_5$.
 \begin{itemize}
 \item If $L_{10}$ passes through $L_3\cap L_5$, then $L_9$ passes through $L_2\cap L_5$ and $L_4\cap L_7$, or $L_2\cap L_7$ and $L_4\cap L_5$.
 \begin{itemize}\item Assume that  $L_2\cap L_5$ and $L_4\cap L_7$ are in $L_9$.  It is not hard to check that such an arrangement can not be realized.
 \item Assume that  $L_2\cap L_7$ and $L_4\cap L_5$ are in $L_9$.  By the permutation (8, 9, 10)(4, 3, 2)(5, 6, 7), the arrangement is isomorphic to the one defined by Equation $\eqref{eq:str}$.

 \end{itemize}
 \item If $L_{10}$ passes through $L_2\cap L_5$, then $L_9$ passes through $L_3\cap L_5$ and $L_4\cap L_7$. One can check that by the permutation (5, 7, 6)(10, 9, 8)(2, 4) the arrangement is isomorphic to the one defined by Equation $\eqref{eq:str}$.
 \end{itemize}
\end{enumerate}
\end{enumerate}

\paragraph{Case 3:} Assume that there are 9 triple points.

Then only 8 triple points are in $L_8\cup L_9\cup L_{10}$. Since we assume that each line passes through at least 3 multiple points,  those 8 triple points should be in $(L_1\cup L_2\cup L_3\cup L_4)\cap (L_8\cup L_9\cup L_{10})$. Moreover,
 at least 1 and at most 2 of the intersection points in $\{L_8\cap L_9, L_8\cap L_{10}, L_9\cap L_{10}\}$ should be triple points. In fact, if  none of  $\{L_8\cap L_9, L_8\cap L_{10}, L_9\cap L_{10}\}$ is a triple point, then
 one of $L_8$, $L_9$ and $L_{10}$ will pass through at most 2 multiple points. If all of $\{L_8\cap L_9, L_8\cap L_{10}, L_9\cap L_{10}\}$ are triple points, then one of $L_5$, $L_6$ and $L_7$ will pass through at most 2 multiple points.
 \begin{enumerate}
 \item Assume that $L_8\cap L_9$ is a triple point. Then $L_{10}$ passes through 3 of the double points in $(L_1\cup L_2\cup L_3\cup L_4)\cap (L_5\cup L_6\cup L_7)$ and each of $L_8$ and $L_9$ should pass through 2 of those double points. We may assume that $L_{10}$ passes through $L_1\cap L_5$, $L_2\cap L_6$ and $L_3\cap L_7$, and that $L_8\cap L_9$ is on $L_3$. Moreover, we may assume that $L_2\cap L_8$ and $L_1\cap L_9$ are triple points.  We can write down the defining equations of some of the lines as follows:
 $L_1=\{y=0\}$, $L_2=\{y=z\}$,  $L_3=\{y=t_1z\}$,  $L_4=\{y=t_2z\}$,  $L_5=\{x=0\}$, $L_6=\{x=z\}$, $L_7=\{x=t_1z\}$ and $L_{10}=\{y=x\}$, where $t_1\neq t_2\in \mathbb{C}\setminus\{0, 1\}$.
We have the following possible cases:
\begin{enumerate}
\item\label{a} If $L_2\cap L_5\in L_8$, $L_4\cap L_6\in L_8$, $L_4\cap L_5\in L_9$ and $L_1\cap L_7\in L_9$,
then $L_8=\{y=(t_2-1)x+z\}$ and $L_9=\{y=-\frac{t_2}{t_1}x+t_2z\}$. Since $L_3\cap L_8\cap L_9$ is a triple points, $t_1$ and $t_2$ satisfy the following equation: $t_1t_2^2-t_1^2t_2-2t_1t_2+t_1^2+t_2=0$ which defines an irreducible curve in $\mathbb{C}^2$.

\item $L_2\cap L_5\in L_8$, $L_4\cap L_6\in L_8$, $L_4\cap L_7\in L_9$ and $L_1\cap L_6\in L_9$, then $L_8=\{y=(t_2-1)x+z\}$ and $L_9=\{y=\frac{t_2}{t_1-1}(x-z)\}$. Since $L_3\cap L_8\cap L_9$ is a triple points, $t_1$ and $t_2$ satisfy the following equation: $t_2^2+t_1^2t_2-2t_1t_2-t_1^2+t_1=0$ which defines an irreducible curve in $\mathbb{C}^2$.

\item \label{c} $L_2\cap L_5\in L_8$, $L_4\cap L_7\in L_8$, $L_4\cap L_6\in L_9$ and $L_1\cap L_7\in L_9$, then $L_8=\{y=\frac{t_2-1}{t_1}x+z\}$ and $L_9=\{-\frac{t_2}{t_1-1}(x-t_1z)\}$, where $t_1$ and $t_2$ satisfy the following equation: $t_2^2-2t_1t_2+t_2+t_1-1=0$ which is irreducible.

\item $L_2\cap L_5\in L_8$, $L_4\cap L_7\in L_8$, $L_4\cap L_5\in L_9$ and $L_1\cap L_6\in L_9$, then $L_8=\{y=\frac{t_2-1}{t_1}x+z\}$ and $L_9=\{-t_2x+t_2z\}$, where $t_1$ and $t_2$ satisfy the following equation: $t_2^2-t_1^2t_2-t_2+t_1=0$ which is irreducible.

\item $L_2\cap L_7\in L_8$, $L_4\cap L_5\in L_8$, $L_1\cap L_7\in L_9$ and $L_4\cap L_6\in L_9$, then $L_8=\{y=-\frac{t_2-1}{t_1}x+t_2z\}$ and $L_9=\{y=-\frac{t_2}{t_1-1}(x-t_1z)\}$, where $t_1$ and $t_2$ satisfy the following equation: $t_1^2-t_1=0$. However, by our assumption, $t_1\in \mathbb{C}\setminus\{0, 1\}$. Therefore, this case can not be realized.

\item $L_2\cap L_7\in L_8$, $L_4\cap L_5\in L_8$, $L_1\cap L_6\in L_9$ and $L_4\cap L_7\in L_9$.  By the permutation $(8, 9)(5, 6)(1, 2)$, we see that an arrangement in this case is lattice isomorphic to an arrangement in case \eqref{c}.

\item $L_2\cap L_7\in L_8$, $L_4\cap L_6\in L_8$, $L_1\cap L_6\in L_9$ and $L_4\cap L_5\in L_9$.  By the permutation $(8, 9)(5, 6)(1, 2)$, we see that an arrangement in this case is isomorphic to a one in case \eqref{a}.


\item $L_2\cap L_7\in L_8$, $L_4\cap L_6\in L_8$, $L_1\cap L_7\in L_9$ and $L_4\cap L_5\in L_9$, then $L_8=\{y=-\frac{t_2-1}{t_1-1}(x-z)+t_2z\}$ and $L_9=\{y=-\frac{t_2}{t_1}x+t_2z)\}$, where $t_1$ and $t_2$ satisfy the following equation:  $t_1^2-2t_1t_2+t_2=0$ which is irreducible.

\end{enumerate}

 \item Assume that $L_8\cap L_9$ and $L_8\cap L_{10}$ are both triple. Since there are 8 triple points in $(L_8\cup L_9\cup L_{10})\cap (L_1\cup L_2\cup L_3\cup L_4)$, including $L_8\cap L_9$ and $L_8\cap L_{10}$, then one of those $L_8$, $L_9$ and $L_{10}$ will pass through 4 triple points. We may assume that either $L_8$ or $L_{10}$ passes through 4 triple points. Moreover, it is not hard to see that each of $L_1$, $L_2$, $\dots$, $L_7$ will pass through exactly 2 triple points so that there will be 3 multiple points on each of them.

 \begin{enumerate}
 \item Assume that $L_8$ passes through 4 triple points. By lattice isomorphism, we may assume that $L_1\cap L_5$, $L_2\cap L_9$, $L_3\cap L_{10}$ and $L_4\cap L_7$ are on $L_8$. Since $L_6$ passes three triple points, then  one of $\{L_1\cap L_9, L_3\cap L_9, L_4\cap L_9\}$
and one of $\{L_4\cap L_{10}, L_2\cap L_{10}, L_1\cap L_{10}\}$ should be triple points on $L_6$.
We can write the equations of $L_1$, $L_2$, $\dots$, $L_8$ as follows:
 $L_1=\{y=0\}$, $L_2=\{y=z\}$,  $L_3=\{y=t_2z\}$,  $L_4=\{y=t_3z\}$,  $L_5=\{x=0\}$, $L_6=\{x=z\}$, $L_7=\{x=t_1z\}$ and $L_{8}=\{y=\frac{t_3}{t_1}x\}$, where $t_1, t_2, t_3\in \mathbb{C}\setminus\{0, 1\}$ and $t_2\neq t_3$.
 Since $L_6$ passes through 3 triple points, up to a lattice isomorphism, we have the following possibilities:
\begin{itemize}
\item $L_1\cap L_6\cap L_9$, $L_4\cap L_6\cap L_{10}$, $L_3\cap L_5\cap L_9$ and $L_2\cap L_{10}\cap L_7$ are triple points.
 Then $L_9=\{y=-t_2x+t_2z\}$ and $L_{10}=\{y=\frac{t_3-1}{1-t_1}(x-t_1z)+z\}$, where $t_1=\frac{t_2t_3-t_3}{t_2}$ and $t_3^2+2t_2t_3+2t_2=0$. Therefore, the moduli space is irreducible.
 \item $L_1\cap L_6\cap L_9$, $L_4\cap L_6\cap L_{10}$, $L_3\cap L_7\cap L_9$ and $L_2\cap L_{10}\cap L_5$ are triple points.
 Then $L_9=\{y=\frac{t_2}{t_1-1}(x-z)\}$ and $L_{10}=\{y=(t_3-1)x+z\}$, where $t_1=\frac{t_2t_3-t_3}{t_2-t_3}$ and $t_3=\frac{2t_2}{t_2+1}$. Therefore, the moduli space is irreducible.
 \item $L_3\cap L_6\cap L_9$, $L_1\cap L_6\cap L_{10}$, $L_4\cap L_5\cap L_9$ and $L_2\cap L_{10}\cap L_7$ are triple points.  Then $L_9=\{y=(t_2-t_3)x+t_3z\}$ and $L_{10}=\{y=\frac{1}{t_1-1}(x-z)\}$, where $t_1=\frac{t_3^2-t_3}{t_3-t_2}$ and $t_2^2+(t_3^2-3t_3)t_2+t_3=0$. Therefore, the moduli space is irreducible.
 \item $L_3\cap L_6\cap L_9$, $L_2\cap L_6\cap L_{10}$, $L_4\cap L_5\cap L_9$ and $L_1\cap L_{10}\cap L_7$ are triple points.
 Then $L_9=\{y=(t_2-t_3)x+t_3z\}$ and $L_{10}=\{y=\frac{1}{1-t_1}(x-t_1z)\}$, where $t_1=\frac{t_3-t_3^2}{t_2-t_3}$ and $t_2=t_3-t_3^2$. Therefore, the moduli space is irreducible.
\end{itemize}

 \item Assume that $L_{10}$ passes through 4 triple points. By lattice isomorphism, we may assume that $L_1\cap L_5$, $L_2\cap L_6$, $L_3\cap L_7$ and $L_4\cap L_8$ are on $L_{10}$. Moreover, we may assume that $L_1\cap L_6\cap L_8$ is a triple point. Since $L_8\cap L_9$ is also a triple points, then $L_8\cap L_9 $ is either on $L_2$ or $L_3$. If $L_2\cap L_8\cap L_9$ is a triple point, then $L_4\cap L_7\cap L_9$, and $L_3\cap L_5\cap L_9$ should be triple points by our assumption.  If $L_3\cap L_8\cap L_9$ is a triple point, then $L_9$ should contain $\{L_4\cap L_7, L_2\cap L_5\}$
 or $\{L_2\cap L_7, L_4\cap L_5\}$. Similar to the case (a), with some elementary algebraic computations, we see that the moduli spaces are irreducible.
 \end{enumerate}
 \end{enumerate}
 \end{proof}

\subsection{All triple points are in the pencil of the quadruple point}
Assume that all the triple points are on the lines passing through the quadruple point. We first show that there are at most 11 triple points so that the arrangement is non-reductive.

\begin{lemma}\label{prop:quadruple+3-3triples}Let $\mathcal{A}$ be a non-reductive arrangement of $10$ lines with $1$ quadruple point so that all  triple points are on the lines passing through the quadruple point. Then
there are at most $11$ triple points.
\end{lemma}
\begin{proof}
Assume that the quadruple point is $L_3\cap L_4\cap L_7\cap L_8$ and there are 12 triple points. Then each of those 4 lines will pass through exactly 3 triple points. However, we will show that it can not be realized. We may assume that $L_1\cap L_2\cap L_3$ and $L_4\cap L_5\cap L_6$ are two triple points.  By lattice isomorphism, we see that the arrangement must contain the following sub-arrangement (see Figure \ref{n_4-0-sub}) so that each of $L_3$ and $L_4$ passes through 3 triple points.
\begin{figure}[htbp]\centering
\begin{tikzpicture}[domain=-2:4, scale=0.6]\tikzstyle{every node}=[font=\footnotesize]
\draw[domain=-2:4] plot (0, \x) node[above]{$L_5$};
\draw[domain=-2:4] plot (1, \x) node[above]{$L_4$};
\draw[domain=-2:4] plot (2, \x) node[above]{$L_6$};

\draw[domain=4:-2] plot (\x, 0) node[left]{$L_2$};
\draw[domain=4:-2] plot (\x, 1) node[left]{$L_3$};
\draw[domain=4:-2] plot (\x, 2) node[left]{$L_1$};

\draw[domain={0.8:1.2}] plot (\x, \x) node[right]{$L_{7}$};
\draw[domain={0.8:1.2}] plot (\x, 2-\x) node[below]{$L_{8}$};
\draw[domain=-2:3] plot (\x, {\x+1}) node[right]{$L_9$};
\draw[domain=-1:4] plot (\x, {(\x-1)}) node[right]{$L_{10}$};
\end{tikzpicture}
\caption{\label{n_4-0-sub}}\end{figure}

Since $L_7$ also passes through 3 triple points, then either $L_7\cap L_9$ or $L_7\cap L_{10}$ should be a triple point. Up to a lattice isomorphism, we may assume that $L_7\cap L_9$ is a triple point.

If  $L_7\cap L_9\cap L_{10}$ is a triple point, then either $\{L_2\cap L_5, L_1\cap L_6\}$ or $\{L_1\cap L_5, L_2\cap L_6\}$ should be in $L_7$ so that $L_7$ will pass through 3 triple points. If $\{L_2\cap L_5, L_1\cap L_6\}\subset L_7$, then the only possible triple points are $L_1\cap L_5$,  $L_2\cap L_6$, $L_2\cap L_9$, $L_6\cap L_9$, $L_1\cap L_{10}$, $L_5\cap L_{10}$. Among any 3 of them, 2 of the 3 points are on the same line. Therefore, $L_8$ can not pass through 3 triple points in this case.
If $\{L_1\cap L_5, L_2\cap L_6\}\subset L_7$, then the only possible triple points are $L_2\cap L_5$,  $L_1\cap L_6$, $L_2\cap L_9$, $L_6\cap L_9$, $L_1\cap L_{10}$, $L_5\cap L_{10}$. Among any 3 of them, 2 of the 3 points are on the same line. Therefore, $L_8$ can not pass through 3 triple points in this case.

 Now consider the case that $L_7\cap L_9\cap L_{10}$ is not a triple point. So either $L_2\cap L_7\cap L_9$ or $L_6\cap L_7\cap L_{9}$ is a triple.
By switching the labels between $L_2$ and $L_6$ and others accordingly, we assume that $L_2\cap L_7\cap L_9$ is a triple point.

 \begin{itemize}
\item If $L_7\cap L_{10}$ is not a triple, then only one more point, $L_1\cap L_5$ or $L_1\cap L_6$ could be a triple point on $L_7$. Thus there will be at most 2 triple points on $L_7$.

\item If $L_7\cap L_{10}$ is a triple point, then it can not be on $L_1$. Otherwise, $L_7$ will have only two triple points $L_1\cap L_7\cap L_{10}$ and $L_2\cap L_7\cap L_9$, because all the possible triple points are on $L_1\cup L_{10}\cup L_2\cup L_9$.

Therefore, $L_5 \cap L_7\cap L_{10}$ is the triple point. And the third triple points on $L_7$ must be $L_1\cap L_6\cap L_7$.  

 Now we can write down equations of lines of the sub-arrangement $\mathcal{A}\setminus\{L_8\}$ as follows: $L_1=\{y=t_2z\}$, $L_2=\{y=-z\}$, $L_3=\{y=0\}$, $L_4=\{x=0\}$, $L_5=\{y=-z\}$, $L_6=\{y=t_1z\}$, $L_9=\{y=t_2x+t_2z\}$, $L_{10}=\{y=\frac{1}{t_1}x-z\}$, and $L_7=\{y=\frac{t_2}{t_1}x\}$, where $t_1, t_2\in\mathbb{C}\setminus\{0, 1\}$. The two variables must satisfy the following equations associated to the triple points:
\begin{align*}
&L_2\cap L_7\cap L_9\neq\emptyset:&  t_1=t_2+1, \\
&L_6\cap L_7\cap L_9\neq\emptyset:& t_2=t_1+1.
 \end{align*}
However, those 2 equations have no common solution. So $L_7$ can not passes through 3 triple points if $L_7\cap L_9\cap L_{10}$ is not a triple point.
\end{itemize}
\end{proof}

The classification will run on numbers of triple points.

\begin{thm}Let $\mathcal{A}$ be a non-reductive arrangement of $10$ lines with $1$ quadruple point and $11$ triple points such that all triple points are on the $4$ lines passing through the quadruple point. Then the moduli space $\mathcal{M}_{\mathcal{A}}$ is irreducible.
\end{thm}
\begin{proof}
Let $L_3\cap L_4\cap L_7\cap L_8$ be the quadruple point. We may assume that each of $L_3$, $L_4$, and $L_7$ passes through 3 triple points and $L_8$ passes through 2 triple points.
As we have seen in the proof of Proposition \ref{prop:quadruple+3-3triples}, $L_7\cap L_9\cap L_{10}$ must be a triple point.  Then either $\{L_2\cap L_5, L_1\cap L_6\}$ or $\{L_1\cap L_5, L_2\cap L_6\}$ is in $L_7$ so that $L_7$ will pass through 3 triple points.
\begin{itemize}
\item $\{L_2\cap L_5, L_1\cap L_6\}\subset L_7$.   By lattice isomorphisms and automorphisms of the dual projective plane, we may assume that the arrangement contains the following sub-arrangement (Figure \ref{n_4=1-11triples}):
\begin{figure}[htbp]\centering
\begin{tikzpicture}[domain=-3:3, scale=0.6]\tikzstyle{every node}=[font=\footnotesize]
\draw[domain=-3:3] plot (0, \x) node[above]{$L_4$};
\draw[domain=-3:3] plot (1, \x) node[above]{$L_5$};
\draw[domain=-3:3] plot (2, \x) node[above]{$L_6$};

\draw[domain=3:-3] plot (\x, 0) node[left]{$L_3$};
\draw[domain=3:-3] plot (\x, 1) node[left]{$L_2$};
\draw[domain=3:-3] plot (\x, 2) node[left]{$L_1$};

\draw[domain={-2.5:3}] plot (\x, \x) node[right]{$L_{7}$};
\draw[domain=-3:3] plot (\x, {-0.5*\x+1}) node[right]{$L_{10}$};
\draw[domain=2.5:-.5] plot (\x, {-2*\x+2}) node[above]{$L_{9}$};
\end{tikzpicture}
\caption{\label{n_4=1-11triples}}
\end{figure}
Defining equations of the lines can be written as: $L_1=\{y=tz\}$, $L_2=\{y=z\}$, $L_3=\{y=0\}$, $L_4=\{x=0\}$, $L_5=\{x=z\}$, $L_6=\{y=tz\}$, $L_7=\{y=x\}$, $L_9=\{y=-tx+tz\}$ and $L_{10}=\{y=-\frac{1}{t}x+z\}$, where $t\in\mathbb{C}$ and $t\neq 0, 1$.

The only possible triple points are $L_1\cap L_5=(1:t:1)$,  $L_2\cap L_6=(t:1:1)$, $L_2\cap L_9=(\frac{t-1}{t}:1:1)$, $L_6\cap L_9=(t:t-t^2:1)$, $L_1\cap L_{10}=(t-t^2:t:1)$, $L_5\cap L_{10}=(1:1-\frac{1}{t}:1)$. Among any three, two of them are on the same line. Since $L_8$ passes through 2 triple points,  the possibilities are as follows:
\begin{enumerate}
\item If $L_8$ passes through $L_1\cap L_5$ and  $L_2\cap L_6$, then $t=-1$.
\item If $L_8$ passes through $L_1\cap L_5$ and $L_2\cap L_9$, then $t=2$.
\item If $L_8$ passes through $L_1\cap L_5$ and $L_6\cap L_9$, then $t=1/2$.
\item If $L_8$ passes through $L_1\cap L_{10}$ and $L_2\cap L_6$, then $t=1/2$.
\item If $L_8$ passes through $L_1\cap L_{10}$ and $L_2\cap L_9$, then $t=-1$.
\item If $L_8$ passes through $L_1\cap L_{10}$ and $L_6\cap L_9$, then $t=2$.
\item  If $L_8$ passes through $L_2\cap L_6$ and $L_5\cap L_{10}$, then $t=2$.
\item If  $L_8$ passes through $L_2\cap L_9$ and $L_5\cap L_{10}$, then $t=1/2$.
\item  If $L_8$ passes through $L_6\cap L_9$ and $L_5\cap L_{10}$, then $t=-1$.
\end{enumerate}

\item If $\{L_1\cap L_5, L_2\cap L_6\}\subset L_7$, then by an automorphism of the dual plan, we can write defining equations of the lines as follows: $L_1=\{y=t_2z\}$, $L_2=\{y=z\}$, $L_3=\{y=0\}$, $L_4=\{x=0\}$, $L_5=\{x=z\}$, $L_6=\{y=t_1z\}$,  $L_9=\{y=-t_2x+t_2z\}$, $L_{10}=\{y=-\frac{1}{t_1}x+z\}$, $L_7=\{y=t_2x\}$, where $t_1, t_2\in \mathbb{C}\setminus\{0, 1\}$. Since $L_7$ passes through $L_2\cap L_6$, then $t_1$, $t_2$ satisfy the following equation: $t_1t_2=1$. Since $L_7\cap L_9\cap L_{10}$ is a triple point, then $t_2=1$. Therefore this case can not be realized.
\end{itemize}

\end{proof}

\begin{thm}\label{prop:n4=1-0-n3=10}
Let $\mathcal{A}$ be a non-reductive arrangement of $10$ lines with $1$ quadruple point and $10$ triple points such that all triple points are on the 4 lines passing through the quadruple point. Then the quotient moduli space $\mathcal{M}^c_{\mathcal{A}}$ is irreducible.
\end{thm}
\begin{proof}Let $L_3\cap L_4\cap L_7\cap L_8$ be the quadruple point.  Since there are 10 triple points on those 4 lines and we know that each of the 4 lines passes through at least 2 and at most 3 triple points, then we may assume that each of $L_3$ and $L_4$ passes through 3 triple points.
On the other hand, each of the other lines passes through at least 3 and at most 4 triple points.  Let $a$ and $b$ be the numbers of lines in $\{L_1, L_2, L_5, L_6, L_9, L_{10}\}$ which pass through 4 and 3 triple points, respectively.  Then $a$ and $b$ should satisfy the following system of equations:
\begin{align*}a+b&=6\\ 4a+3b&=20.
\end{align*}
It follows that $a=2$ and $b=4$.  Assume that $L$ and $L'$ are the two lines such that each passes through 4 triple points. Then either $L\cap L'$ is a triple point on $L_3\cup L_4$, or not. If $L\cap L'$ is a triple points on $L_3\cup L_4$, we can assume that $L=L_1$ and $L'=L_2$ and $L\cap L'$ is on $L_3$. If $L\cap L'$ is not a triple point, we can assume that $L=L_1$, $L'\in\{L_5, L_6\}$. Moreover, we may assume that $L_1\cap L_2\cap L_3$ is a triple point and $L_4\cap L_5\cap L_6$ is a triple point.

Up to the permutation $(5, 6)$,  we may assume the arrangements contain the following sub-arrangement (Figure \ref{n_4-1-10triples}).
\begin{figure}[htbp]\centering
\begin{tikzpicture}[domain=-3:3, scale=0.6]\tikzstyle{every node}=[font=\footnotesize]
\draw[domain=-3:3] plot (0, \x) node[above]{$L_4$};
\draw[domain=-3:3] plot (1, \x) node[above]{$L_5$};
\draw[domain=-3:3] plot (2, \x) node[above]{$L_6$};

\draw[domain=3:-3] plot (\x, 0) node[left]{$L_3$};
\draw[domain=3:-3] plot (\x, 1) node[left]{$L_2$};
\draw[domain=3:-3] plot (\x, 2) node[left]{$L_1$};

\draw[domain={0.25:-0.25}] plot (\x, 2*\x) node[below]{$L_{8}$};
\draw[domain={0.5:-0.5}] plot (\x, 0.5*\x) node[left]{$L_{7}$};
\draw[domain=-3:3] plot (\x, {-0.5*\x+1}) node[right]{$L_{10}$};
\draw[domain=2.5:-.5] plot (\x, {-2*\x+2}) node[above left]{$L_{9}$};
\end{tikzpicture}
\caption{\label{n_4-1-10triples}}
\end{figure}
We can write down the defining equations of the lines as follows: $L_1=\{y=t_2z\}$, $L_2=\{y=z\}$, $L_3=\{y=0\}$, $L_4=\{x=0\}$, $L_5=\{x=z\}$, $L_6=\{y=t_1z\}$,  $L_9=\{y=-t_2x+t_2z\}$, and $L_{10}=\{y=-\frac{1}{t_1}x+z\}$, where $t_1, t_2\in \mathbb{C}\setminus\{0, 1\}$.
It is not hard to check that that $L_7$ or $L_8$ will pass through an extra triple point $L_9\cap L_{10}$, if $L_7$ or $L_8$ passes through $L_1\cap L_6$ and $L_2\cap L_5$. So we should assume that $L_1\cap L_6$ and $L_2\cap L_5$ are not on $L_7$ or $L_8$.  \begin{enumerate}
\item Assume that $L_2$ also passes through 4 triple points.  It is not hard to see that $L_9\cap L_{10}$ can not be a triple point on $L_7\cup L_8$, otherwise, either $L_1$ or $L_2$ will pass through at most 3 triple points.  Thus $L_1\cap L_{10}$ and $L_2\cap L_9$ should be triple points on $L_7\cup L_8$.
 By lattice isomorphisms, we may assume that $L_7$ passes $L_2\cap L_9$. Since there are 4 triple points on $L_1$, $L_1\cap L_7$ should be a triple point.
\begin{enumerate}
\item Assume that
 $L_7$ passes through $L_1\cap L_5$. Then $L_7=\{y=t_2x\}$ and $t_2=2$. By assumption,  $L_8$ should pass through $L_2\cap L_6$ and $L_1\cap L_{10}$. Then $L_8=\{y=\frac{1}{t_1}x\}$. The fact that $L_1\cap L_8\cap L_{10}$ is a triple point implies that $t_2=\frac{1}{2}$. Therefore, this case can not be realized.
\item Assume that  $L_7$ passes through $L_1\cap L_6$. Then $L_8$ passes through $L_2\cap L_5$ and $L_1\cap L_{10}$. Therefore,  $t_1=t_2-1$ and $t_2+t_1(t_2-1)=0$. It follows that $t_2^2-t_2+1=0$ and $t_1=t_2-1$. The defining equation of the arrangement can be written as the following
\begin{equation*}\label{eq:n_4=1-0-n_3=10-1}
\scriptstyle
xy(x-z)(y-z)(x-(t-1)z)(y-t z)(y-\frac{t}{t-1}x)(y-x)(y+t x-t z)(y+\frac{1}{t-1}x-z),
\end{equation*}
where $t=\frac{1\pm\sqrt{-3}}{2}$.
\item Assume that  $L_7$ passes through $L_1\cap L_{10}$. Then $L_8$ passes through $L_1\cap L_5$ and $L_2\cap L_6$. Therefore,  $t_1=-1$ and  $t_2=-1$.
\end{enumerate}

\item Assume that $L_5$ passes through 4 triple points.  We claim that $L_9\cap L_{10}$ can not be on $L_7\cup L_8$. Assume, in contrary,  that $L_9\cap L_{10}$ is on $L_7$, then $L_7$ must pass through $L_1\cap L_5$ and $L_8$ must pass through $L_2\cap L_5$ and $L_1\cap L_6$ by assumptions.  However,  if $L_8$ passes through $L_2\cap L_5$ and $L_1\cap L_6$, then it also passes through $L_9\cap L_{10}$.  That is impossible. On the other hand,  if none of $L_7$ and $L_8$ passes through $L_1\cap L_5$, then $L_7\cup L_8$ must pass through $L_2\cap L_5$ and $L_{10}\cap L_5$ so that $L_5$ passes through 4 triple points.  Consequently, $L_7$ and $ L_8$ should also pass through $L_1\cap L_{10}$ and $L_1\cap L_6$ so that $L_1$ will pass through 4 triple points. However, we notice that there will be 4 triple points on $L_{10}$. That contradicts our assumption.
Therefore, we may assume that $L_7$ passes through $L_1\cap L_5$ but not $L_9\cap L_{10}$. Since there are 2 triple points on $L_7$, there should be one more triple point on $L_7$. Here are the possibilities:
\begin{enumerate}
\item Assume that
 $L_7$ passes $L_2\cap L_9$. Then $L_8$ passes through $L_1\cap L_6$ and $L_5\cap L_{10}$  by assumption. Therefore $t_2=2$ and $t_1=3$.
\item Assume that  $L_7$ passes $L_2\cap L_6$. Then $L_8$ must pass through 1 point of each of the sets $\{L_1\cap L_{10}, L_1\cap L_6\}$ and $\{L_2\cap L_5, L_5\cap L_{10}\}$. However, since none of those 4 points are on $L_9$, then $L_9$ will pass through only 2 triple points.
\item Assume that  $L_7$ passes through $L_9\cap L_6$. Then $L_8$ passes through $L_1\cap L_{10}$ and $L_2\cap L_5$. Therefore, $t_1=\frac{1}{2}$ and $t_2=\frac{1}{3}$.
\end{enumerate}

\item Assume that  $L_6$ passes through 4 triple points. 
We claim that $L_1\cap L_6$ can not be on $L_7\cup L_8$,  if $L_9\cap L_{10}$ is not on $L_7\cup L_8$.  Assume, in the contrary,  that $L_9\cap L_{10}$ is not on $L_7\cup L_8$, but $L_1\cap L_6$ is on $L_7\cup L_8$. We may assume that $L_7$ passes $L_1\cap L_6$. Then up to a lattice isomorphism, we may also assume that  $L_2\cap L_9$ is on $L_7$. Then $L_8$ must pass through one of $L_2\cap L_6$ and $L_9\cap L_6$ so that $L_6$ passes through 4 triple points. However, there are already 3 triple points on each $L_2$ and $L_9$ and we assume that each of them passes through exactly 3 triple points. We obtain a contradiction.

Assume that $L_1\cap L_6$ and  $L_9\cap L_{10}$ are both also on $L_7$. Then $L_8$ should pass through $L_1\cap L_5$ and $L_2\cap L_6$. By writing down the defining equation, we see that this can not be realizable.

Assume that $L_1\cap L_6$ is on $L_7$ and $L_9\cap L_{10}$ is on $L_8$. Since there are only 2 triple points on $L_8$, then 1 of $L_1$ and  $L_6$ will pass through at most 3 triple points.

Therefore, none of $L_9\cap L_{10}$ and $L_1\cap L_6$ should be on $L_7\cup L_8$.  Up to lattice isomorphisms, there are only  two possible cases:
\begin{enumerate}
\item  $L_7$ passes $L_1\cap L_{10}$ and $L_2\cap L_6$ and $L_8$ passes through $L_1\cap L_5$ and $L_6\cap L_9$. It follows that $t_1=\frac{1}{2}$ and $t_2=\frac{1}{2}$.

\item  $L_7$ passes $L_1\cap L_5$ and $L_2\cap L_6$ and $L_8$ passes through $L_1\cap L_{10}$ or $L_6\cap L_9$. It follows that $t_1t_2=1$ and $t_1+t_2=1$. The defining equation can be written as
 \begin{equation*}\label{eq:n_4=1-0-n_3=10-3}
\scriptstyle
xy(x-z)(y-z)(x-(1-t)z)(y-t z)(y+x)(y-t x)(y+t x-t z)(y+\frac{1}{1-t}x-z),
\end{equation*}
where $t=\frac{1\pm\sqrt{-3}}{2}$.
\end{enumerate}

\end{enumerate}

\end{proof}

\begin{thm}\label{prop:n4-1-3triples}
Let $\mathcal{A}$ be a non-reductive arrangement of $10$ lines with $1$ quadruple point and $9$ triple points such that  all triple points are on the $4$ lines passing through the quadruple point. Then the moduli space $\mathcal{M}_{\mathcal{A}}$ is irreducible.
\end{thm}
\begin{proof}
Let $L_3\cap L_4\cap L_7\cap L_8$ be the quadruple point. By the assumption, we may assume that $L_3$ passes through 3 triple points and each of $L_4$, $L_7$ and $L_8$ passes through 2 triple points. Similar to the proof of Propositioin \ref{prop:n4=1-0-n3=10}, by our assumption,  each of the other 6 lines, $L_1$, $L_2$, $L_5$, $L_6$, $L_9$ and $L_{10}$, should pass exactly 3 triple points. Let $a$ and $b$ be the numbers of lines in $\{ L_1, L_2, L_5, L_6, L_9, L_{10}\}$ passing through 4 and 3 triple points, respectively.  Then $a$ and $b$ should satisfy the following system of equations:
\begin{align*}a+b&=6\\ 4a+3b&=18.\end{align*}
It follows that $a=0$ and $b=6$.

 Let $L_1\cap L_2$ be 1 of the triple points on $L_3$. Then there are 7 triple points on $L_1\cup L_2\cup L_3$. We may assume $L_4\cap L_5\cap L_6$ is a triple point apart from $L_1\cup L_2\cup L_3$. The other triple point apart from $L_1\cup L_2\cup L_3$ must be in $(L_9\cup L_{10})\cap (L_4\cup L_7\cup L_8)$.  Since there are only 2 triple points on $L_4$, we may assume that $L_2\cap L_4$ is not a triple point. Then $L_2\cap L_7$ and $L_2\cap L_8$ should be triple points so that $L_2$ will pass through 3 triple points and all triple points are on $L_3\cup L_4\cup L_7\cup L_8$. Similarly, $L_3\cap L_9$ and $L_3\cap L_{10}$ should be triple points so that $L_3$ will pass through 3 triple points.  Moreover, we may assume that $L_3\cap L_9$ and $L_3\cap L_{10}$ are in $L_5$ and $L_6$, respectively.

\begin{enumerate}
\item\label{1-irr} $L_2\cap L_7$ and $L_2\cap L_8$ are not on $L_5\cup L_6$. We may assume that $L_2\cap L_7\cap L_9$ and $L_2\cap L_8\cap L_{10}$ are triple points.  Recall  that, by the assumption,  there is another triple point in $(L_9\cup L_{10})\cap ((L_4\cup L_7\cup L_8)\setminus(L_1\cup L_2\cup L_3))$. Then $L_4\cap L_9\cap L_{10}$ is a triple point.  By the assumption, $(L_7\cup L_8)\cap (L_5\cup L_6)\cap L_1$ should be triple points. So there are two choices: $L_1\cap L_5\in L_7$ and $L_1\cap L_6\in L_8$ (see Figure \ref{n_4=1-0-n_3=9-1} (a)), or $L_1\cap L_6\in L_7$ and $L_1\cap L_5\in L_8$ (see Figure \ref{n_4=1-0-n_3=9-1} (b)). In each case, the moduli space $\mathcal{M}_{\mathcal{A}}$ is irreducible.
\begin{figure}[htbp]
\centering
\subfigure[]{
\begin{tikzpicture}[domain=-2:4, scale=0.6]\tikzstyle{every node}=[font=\footnotesize]
\draw[domain=-1:4] plot (0, \x) node[above]{$L_4$};
\draw[domain=-1:4] plot (1, \x) node[above]{$L_5$};
\draw[domain=-1:4] plot (2.5, \x) node[above]{$L_6$};

\draw[domain=3:-1] plot (\x, 0) node[left]{$L_3$};
\draw[domain=3:-1] plot (\x, 1) node[left]{$L_2$};
\draw[domain=3:-1] plot (\x, 1.5) node[left]{$L_1$};
\draw[domain=2.7:-0.5] plot (\x, 1.5*\x) node[below]{$L_7$};
\draw[domain=-1:3] plot (\x, 3/5*\x) node[above]{$L_8$};

\draw[domain=-0.33:1.33] plot (\x, {-3*\x+3}) node[right]{$L_9$};
\draw[domain=3:-5/6] plot (\x, {-6/5*\x+3}) node[left]{$L_{10}$};
\end{tikzpicture}}
\hspace{2cm}
\subfigure[]{
\begin{tikzpicture}[domain=-2:4, scale=0.6]\tikzstyle{every node}=[font=\footnotesize]
\draw[domain=-1:4] plot (0, \x) node[above]{$L_4$};
\draw[domain=-1:4] plot (1, \x) node[above]{$L_5$};
\draw[domain=-1:4] plot (-1, \x) node[above]{$L_6$};

\draw[domain=3:-3] plot (\x, 0) node[left]{$L_3$};
\draw[domain=3:-3] plot (\x, 1) node[left]{$L_2$};
\draw[domain=3:-3] plot (\x, 2) node[left]{$L_1$};
\draw[domain=-0.5:1.5] plot (\x, 2*\x) node[above right]{$L_7$};
\draw[domain=0.5:-1.5] plot (\x, -2*\x) node[above left]{$L_8$};

\draw[domain=-3:2] plot (\x, {-2/3*(\x-1)}) node[right]{$L_9$};
\draw[domain=3:-2] plot (\x, {2/3*(\x+1)}) node[left]{$L_{10}$};
\end{tikzpicture}
}
\caption{\label{n_4=1-0-n_3=9-1}}
\end{figure}

\item Both $L_2\cap L_7$ and $L_2\cap L_8$ are on $L_5\cup L_6$. We may assume that $L_2 \cap L_7\cap L_5$ and $L_2\cap L_8\cap L_6$ are triple points. Then $L_1\cap (L_7\cup L_8)\cap (L_9\cup L_{10})$ must be triple points so that $L_1$ passes through 3 triple points. It follows that  $L_9\cap L_{10}\cap L_4$ is a triple point. By switching labels of $L_1$ and $L_2$ and others accordingly,  the arrangements are lattice isomorphic to the arrangements in the case \ref{1-irr}. Hence the moduli space is irreducible.

\item Only one of $L_2\cap L_7$ and $L_2\cap L_8$ is on $L_5\cup L_6$. We may assume that $L_2 \cap L_7\cap L_5$ is a triple point. Let us consider the possible triple points on $L_7$ and $L_8$.
\begin{enumerate}
\item $L_7$ also passes through $L_1\cap L_6$. Then $L_8\cap (L_1\cup L_2)\cap (L_9\cup L_{10})$ and $L_4\cap L_9\cap L_{10}$ must be triple points by our assumption. Then either $L_8\cap L_2\cap L_9$ or $L_8\cap L_2\cap L_{10}$ is a triple point and correspondingly, $L_8\cap L_1\cap L_{10}$ or $L_8\cap L_1\cap L_9$ is a triple point. One can check that the first choice can not be realizable and the moduli space $\mathcal{M}_{\mathcal{A}}$ in the second choice (see Figure \ref{n_4=1-0-n_3=9-1-3-1}) is irreducible.
\begin{figure}[htbp]
\centering
\begin{tikzpicture}[domain=-2:4, scale=0.6]\tikzstyle{every node}=[font=\footnotesize]
\draw[domain=-1:3] plot (0, \x) node[above]{$L_4$};
\draw[domain=-1:3] plot (1, \x) node[above]{$L_5$};
\draw[domain=-1:3] plot (2, \x) node[above]{$L_6$};

\draw[domain=3:-3] plot (\x, 0) node[left]{$L_3$};
\draw[domain=3:-3] plot (\x, 1) node[left]{$L_2$};
\draw[domain=3:-3] plot (\x, 2) node[left]{$L_1$};
\draw[domain=-1:3] plot (\x, \x) node[above]{$L_7$};
\draw[domain=0.8:-3] plot (\x, -\x) node[above]{$L_8$};

\draw[domain=2.2:-3] plot (\x, {-2/3*(\x-1)}) node[left]{$L_9$};
\draw[domain=-2.8:3] plot (\x, {-1/3*(\x-2)}) node[below]{$L_{10}$};
\end{tikzpicture}
\caption{\label{n_4=1-0-n_3=9-1-3-1}}
\end{figure}
\item $L_7$ does not pass through $L_1\cap L_6$, but $L_8$ passes through $L_1\cap L_6$. Then each of $L_7$ and $L_8$ should pass through 1 more triple point.
\begin{itemize}
\item Both $L_1\cap L_7$ and $L_2\cap L_8$ should be on $L_9\cup L_{10}$. Then $L_4\cap L_9\cap L_{10}$ should be a triple point. Then there are two possibilities: $L_8\cap L_2\cap L_9$ and $L_7\cap L_1\cap L_{10}$ are triple points;  $L_8\cap L_2\cap L_{10}$ and $L_7\cap L_1\cap L_9$ are triple points.  In the first case, $\mathcal{M}_\mathcal{A}$ is irreducible. In the second case, the arrangement can not be realizable.
\item One of $L_1\cap L_7$ and $L_2\cap L_8$ is not on $L_9\cup L_{10}$. Up to a permutation (1, 2)(5, 6)(7, 8)(9 ,10), we may assume that $L_1\cap L_7$ is on $L_9\cap L_{10}$ but $L_2\cap L_8$ is not on $L_9\cap L_{10}$. Then $L_8\cap L_9\cap L_{10}$ must be a triple point. If $L_1\cap L_7$ is on $L_9$, then $L_{10}$ must pass through $L_2\cap L_4$. If $L_1\cap L_7$ is on $L_{10}$, then $L_9$ must pass through $L_2\cap L_4$. It is not difficult to check that the moduli space $\mathcal{M}_\mathcal{A}$ is irreducible.
\end{itemize}
\item Neither $L_7$ nor $L_8$ passes through $L_1\cap L_6$. Then $L_{10}$ should pass through 1 of $L_8\cap (L_1\cup L_2)$ so that $L_8$ will pass through 2 triple points.
\begin{itemize}
\item $L_{10}$ passes through $L_2\cap L_8$. Then $L_9$ passes through $L_6\cap L_8$ or $L_1\cap L_8$.
\begin{itemize}
\item $L_9$ passes through $L_6\cap L_8$. Then $L_{10}$ passes through $L_1\cap L_4$ or $L_1\cap L_7$, correspondingly $L_9$ passes through $L_1\cap L_7$ or $L_1\cap L_4$. In both case, the moduli spaces are irreducible.
\item $L_9$ passes through $L_1\cap L_8$. Then $L_6\cap L_7$ must be on $L_9$. Consequently, $L_{10}$ passes through $L_1\cap L_4$. Again, the moduli space is irreducible.
\end{itemize}
\item $L_{10}$ passes through $L_1\cap L_8$. Then $L_9$ passes through $L_6\cap L_8$ or $L_2\cap L_8$.
\begin{itemize}
\item $L_9$ passes through $L_6\cap L_8$. Then $L_9$ should also pass through $L_1\cap L_7$ so that $L_7$ passes through 2 triple points. Consequently, $L_{10}$ passes through $L_2\cap L_4$. However, this can not be realizable.
\item $L_9$ passes through $L_2\cap L_8$. Then $L_9$ must pass through $L_6\cap L_7$ so that each of $L_6$ and $L_7$ passes through 2 triple points. Consequently, $L_1$ will pass through only 2 triple points $L_1\cap L_2 \cap L_3$ and $L_1\cap L_8\cap L_{10}$. However, we assume that the arrangement $\mathcal{A}$ is non-reductive.
\end{itemize}
\end{itemize}
\end{enumerate}
\end{enumerate}

Therefore, we conclude that the moduli spaces of arrangement under the assumption are irreducible.

\end{proof}

\begin{prop}
Let $\mathcal{A}$ be an arrangement of $10$ lines with $1$ quadruple point and $8$ triple points so that all triple points are on the $4$ lines passing through the quadruple point. Then $\mathcal{A}$ is not non-reductive.
\end{prop}
\begin{proof}$L_1\cap L_2\cap L_3\cap L_4$ be the quadruple point. By assumption, there are 8 triple points.
Let $a$ and $b$ be the number of lines in $\{ L_5, L_6, \dots, L_{10}\}$ passing through 4 and 3 triple points respectively. If each of the 6 lines  passes at least  3 triple points,  then $a$ and $b$ should satisfy the following system of equations:
\begin{align*}a+b&=6\\ 4a+3b&=16.\end{align*}
However, there is no non-negative solution. It follows that at least one of the 6 lines passes at most two triple points.
\end{proof}

By reviewing this subsection, we can make the following conclusion.
\begin{cor}Let $\mathcal{A}$ be an arrangement of $10$ projective lines such that $n_4=1$ and $n_r=0$ for $r\geq 5$. Assume that all triple points are on the lines passing through the quadruple point. If $\mathcal{A}$ does not contain a Falk-Sturmfels arrangement, then the fundamental group $\pi_1(M(\mathcal{A}))$ is determined by the intersection lattice $L(\mathcal{A})$.
\end{cor}

\bibliographystyle{alpha}


{\noindent
Meirav Amram\\ Emmy Noether Research Institute for Mathematics, Bar-Ilan University, Ramat-Gan, 52900, Israel, and\\ Shamoon College of Engineering, Bialik/Basel Sts., Beer-Sheva 84100, Israel\\
meirav@macs.biu.ac.il,~  meiravt@sce.ac.il\\
}

{\noindent Mina Teicher\\
Emmy Noether Research Institute for Mathematics, Bar-Ilan University, Ramat Gan, 52900, Israel, and\\
Institute for Advanced Study,  Einstein Drive, Princeton, NJ., 08540, USA\\
 teicher@macs.biu.ac.il\\
}

{\noindent Fei Ye\\
Emmy Noether Research Institute for Mathematics, Bar-Ilan University, Ramat Gan, 52900, Israel\\
 {\it Current address}: Department of Mathematics, The University of Hong Kong, Pokfulam, Hong Kong\\
  fye@maths.hku.hk}

\end{document}